\documentclass[12pt]{article}
\usepackage{amsmath, amsthm}\usepackage{enumerate}\usepackage{amssymb}
\usepackage{bbold}
\usepackage{bbm}\usepackage{dsfont}\usepackage{color}\usepackage{xcolor}
\usepackage[explicit]{titlesec}
\newtheorem{theorem}{Theorem}[subsection]
\newtheorem{proposition}[theorem]{Proposition}

\newtheorem{lemma}[theorem]{Lemma}
\newtheorem{example}[theorem]{Example}
\newtheorem{corollary}[theorem]{Corollary}

\newtheorem{assertion}[theorem]{Assertion}

\newtheorem{definition}[theorem]{Definition}

\makeatletter
\renewcommand{\subsection}{\@startsection{subsection}{1}
{0pt}{3.25ex plus 1ex minus.2ex}{-1em}{\normalfont\normalsize\bf}}
\makeatother
\begin{document}

\title{{\bf Regularly limited, Grothendieck, 
and Dunford--Pettis operators}}\date{}
\maketitle
\author{\centering{{Safak Alpay$^{1}$, Eduard Emelyanov$^{2}$, Svetlana Gorokhova $^{3}$\\ 
\small $1$ Middle East Technical University, Ankara, Turkey\\ 
\small $2$ Sobolev Institute of Mathematics, Novosibirsk, Russia\\ 
\small $3$ Uznyj matematiceskij institut VNC RAN, Vladikavkaz, Russia}
\abstract{We introduce and study the enveloping norms of regularly 
$\cal{P}$-operators, where $\cal{P}$ is an "almost" 
version of limited, Grothendieck, and of Dunford--Pettis 
operators in Banach lattices. Several further topics related to these 
operators are investigated.}
\vspace{5mm}

{\bf Keywords:} Banach lattice, regularly $\cal{P}$-operator, enveloping norm,
limited operator, Grothendieck operator, Dunford--Pettis operator.

{\bf MSC2020:} {\normalsize 46B25, 46B42, 46B50, 47B60}

}}
\bigskip
\bigskip

Various types of compactness in Banach spaces were in streamline 
of functional analysis in the second half of the 20th century.  
At some point, in order to diversify the classes of compact operators, 
the Banach lattice structure had come in play by Meyer-Nieberg, 
Dodds, Fremlin, Aliprantis, Wickstead, and others.
In the present paper, we continue this line of study with main 
emphasis on regularity of operators under the investigation.
We introduce enveloping norms on spaces of regularly ${\cal P}$-operators
and give conditions under which these spaces are complete under their
enveloping norms.

%%%%%%%%%%%%%%%%%%%%
\section{Introduction and Preliminaries}
%%%%%%%%%%%%%%%%%%%%

Throughout the paper, all vector spaces are real and all operators are linear; $X$, $Y$, and $Z$ 
denote Banach spaces; $E$, $F$, and $G$ denote Banach lattices. A subset $A$ of $X$ is called
{\em bounded} if $A$ is norm bounded. We denote by $B_X$ the closed unit ball of $X$,
by ${\cal L}(X,Y)$ ($\text{\rm L}(X,Y)$, $\text{\rm W}(X,Y)$, $\text{\rm K}(X,Y)$) 
the space of all (continuous, weakly compact, compact) operators 
from $X$ to $Y$, by $E_+$ the positive cone of $E$, by $\text{sol}(A)$ the solid hull 
of $A\subseteq E$, and by $E^a=\{x\in E: |x|\ge x_n\downarrow 0\Rightarrow\|x_n\|\to 0\}$ 
the $\text{\rm o}$-continuous part of $E$. 

\subsection{}
A bounded subset $A$ of $X$ is called
{\em limited} \cite{Diestel}  
(resp. {\em Dunford--Pettis} or a \text{\rm DP}-{\em set} \cite{Andr}) if each 
\text{\rm w}$^\ast$-null (resp. \text{\rm w}-null) sequence in $X'$ is 
uniformly null on $A$. Each limited set is a \text{\rm DP}-set.
All subsets of the {\em closed absolute convex hull} \ 
$\overline{\text{\rm aco}}(A_1+A_2)$ 
of $A_1+A_2$ are limited (resp. \text{\rm DP}) whenever 
$A_1$ and $A_2$ are limited (resp. \text{\rm DP}).
The closed unit ball $B_{c_0}$ of ${c_0}$ is a limited subset of
$\ell^\infty$ by Phillip's Lemma (cf. \cite[Thm.4.67]{AlBu}).
A subset $A$ of $X$ is limited iff $T(A)$ is relatively compact for  
every continuous operator $T:X\to c_0$ (cf. \cite[p.56]{BD}).
In particular, each relatively compact subset $A$ of $X$ is limited. 
By \cite[Thm.5.98]{AlBu}, 
a bounded subset $A$ of $X$ is \text{\rm DP} iff $T(A)$ is a relatively compact 
subset of $Y$ for each $T\in\text{\rm W}(X,Y)$.
A bounded subset $B$ of $X'$ is called an \text{\rm L}-{\em set} 
if each \text{\rm w}-null sequence in $X$ is uniformly null on $B$ (cf. \cite{Ghenciu}).

\subsection{}
We include a proof of the following certainly well
known elementary fact, for which we did not find an appropriate reference.

\begin{proposition}\label{uniform convergence on A}
Let $A\subseteq X$ and $B\subseteq X'$ be nonempty subsets. Then$:$
\begin{enumerate}[{\em i)}] 
\item
A sequence $(f_n)$ in $X'$ is uniformly null on $A$
iff $f_n(a_n)\to 0$ for each sequence $(a_n)$ in $A$. 
\item
A sequence $(x_n)$ in $X$ is uniformly null on $B$
iff $b_n(x_n)\to 0$ for each sequence $(b_n)$ in $B$.
\end{enumerate}
\end{proposition}

\begin{proof}
i) The necessity is obvious. 
Let $f_n(a_n)\to 0$ for each $(a_n)$ in $A$. Suppose
$\limsup\limits_{n\to\infty}(\sup\limits_{a\in A}|f_n(a)|)\ge 3\varepsilon>0$. 
Choose an increasing sequence $(n_k)$ satisfying
$\sup\limits_{a\in A}|f_{n_k}(a)|\ge 2\varepsilon$ for all $k\in\mathbb{N}$, and pick
$a_{n_k}\in A$ with $|f_{n_k}(a_{n_k})|\ge\varepsilon$ for each $k$. Letting
$a_n:=a_{n_1}$ for all $n\in\mathbb{N}\setminus\{n_k: k\in\mathbb{N}\}$ gives
$f_n(a_n)\not\to 0$. The obtained contradiction proves that
$(f_n)$ is uniformly null on $A$.\\
ii) The proof is similar.
\end{proof}
\noindent
The next fact is a direct consequence of 
Proposition \ref{uniform convergence on A}.

\begin{assertion}\label{limited and DP sets}
Let $A$ be a bounded subset of $X$. Then$:$
\begin{enumerate}[{\em (i)}]
\item
$A$ is limited iff $f_n(a_n)\to 0$ for all 
\text{\rm w}$^\ast$-null $(f_n)$ in $X'$ and all $(a_n)$ in $A$.
\item
$A$ is \text{\rm DP} iff $f_n(a_n)\to 0$ for all  
\text{\rm w}-null $(f_n)$ in $X'$ and all $(a_n)$ in $A$.
\end{enumerate}
A bounded subset $B$ of $X'$ is an \text{\rm L}-set
iff $b_n(x_n)\to 0$ for all $(b_n)$ in $B$ and all 
\text{\rm w}-null $(x_n)$ in $X$.
\end{assertion}

\begin{definition}\label{Main Schur property} 
{\em A Banach space $X$ has:
\begin{enumerate}[a)]
\item 
the {\em Schur property} (or $X\in\text{\rm (SP)}$)
if each \text{\rm w}-null sequence in $X$ is norm null (cf. \cite[p.207]{AlBu});
\item 
the {\em Grothendieck property} (or $X\in\text{\rm (GP)}$) 
if each \text{\rm w}$^\ast$-null sequence in $X'$ is \text{\rm w}-null (cf. \cite[p.760]{Wnuk3});
\item
the {\em Dunford--Pettis property}  (or $X\in\text{\rm (DPP)}$) if $f_n(x_n)\to 0$ for each 
\text{\rm w}-null $(x_n)$ in $X$ and each \text{\rm w}-null $(f_n)$ in $X'$ \cite{Gro};
\item 
the {\em Gelfand--Phillips property} (or $X\in\text{\rm (GPP)}$)
if all limited subsets of $X$ are relatively compact (cf. \cite{DU,BD}).
\end{enumerate}}
\end{definition}
\noindent
By \cite{BD}, all separable and all reflexive Banach spaces are in \text{\rm (GPP)}.
It was mentioned in \cite{CCJ} with referring to \cite{Buh} that
a Dedekind $\sigma$-complete Banach lattice $E$ is in $\text{\rm (GPP)}$ 
iff $E$ has $\text{\rm o}$-continuous norm.
In particular, $c_0,\ell^1\in\text{\rm (GPP)}$ 
but $\ell^\infty\not\in\text{\rm (GPP)}$. 

\vspace{3mm}
\noindent
Let ${\cal P}\subseteq {\cal L}(E,F)$. We call elements of ${\cal P}$ by
${\cal P}$-operators and say that ${\cal P}$-operators satisfy the 
{\em domination property} if, 
for each positive $T\in{\cal P}$, it follows from $0\le S\le T$
that $S\in {\cal P}$. An operator $T\in{\cal L}(E,F)$ is called
${\cal P}$-{\em dominated} if $\pm T\le U$ for some  $U\in{\cal P}$.
It is easy to see that, under the assumption ${\cal P}\pm{\cal P}\subseteq{\cal P}\ne\emptyset$,
${\cal P}$-operators satisfy the domination property iff each 
${\cal P}$-dominated operator lies in ${\cal P}$.

\subsection{}
The next three well known facts will be used in the present paper.

\begin{assertion}\label{E' is o-cont} 
{\rm (\cite[Thm.4.59]{AlBu}; \cite[2.4.14]{Mey})
The following are equivalent. 
\begin{enumerate}[(i)] 
\item 
The norm on $E'$ is \text{\rm o}-continuous.
\item 
$E'$ is a KB space.
\item 
Every disjoint bounded sequence in $E$ is \text{\rm w}-null.
\end{enumerate}}
\end{assertion}

\begin{assertion}\label{Dodds-Fremlin}
{\rm (\cite[Cor.2.6, Cor.2.7]{DF})
The following holds. 
\begin{enumerate}[(i)]
\item 
A sequence $(x_n)$ in $E$ is norm null iff $|x_n|$ is $\text{\rm w}$-null  
and $f_n(x_n)\to 0$ for each disjoint bounded $(f_n)$ in $E'_+$.
\item 
A sequence $(f_n)$ in $E'$ is norm null iff $|f_n|$ is $\text{\rm w}^\ast$-null  
and $f_n(x_n)\to 0$ for each disjoint bounded $(x_n)$ in $E_+$.
\end{enumerate}}
\end{assertion}

\begin{assertion}\label{Burkinshaw--Dodds} 
{\rm (cf. \cite[Thm.5.63]{AlBu})
Let $A\subseteq E$ and $B\subseteq E'$ be nonempty bounded sets. 
The following are equivalent.
\begin{enumerate}[(i)]
\item 
Every disjoint sequence in $\text{\rm sol}(A)$ is uniformly null on $B$.
\item 
Every disjoint sequence in $\text{\rm sol}(B)$ is uniformly null on $A$.
\end{enumerate}}
\end{assertion}

\subsection{}
Recall that a Banach lattice $E$ (resp. $E'$) has {\em sequentially} 
\text{\rm w}-{\em continu\-ous} (resp. \text{\rm w}$^\ast$-{\em continuous}) 
{\em lattice operations} if, for each \text{\rm w}-null $(x_n)$ in $E$
(resp. \text{\rm w}$^\ast$-null $(x_n)$ in $E'$), the sequence $(|x_n|)$ 
is \text{\rm w}-null (resp. \text{\rm w}$^\ast$-null). The following
important fact is a direct consequence of \cite[Thm.4.34]{AlBu}.

\begin{assertion}\label{disj w-null is mod w-null}
For every disjoint $\text{\rm w}$-null $(x_n)$ in $E$, 
the sequence $(|x_n|)$ is also \text{\rm w}-null.
\end{assertion}
\noindent
This is no longer true for the $\text{\rm w}^\ast$-topology;
indeed, take $f_n:=e_{2n}-e_{2n+1}\in c'$, then 
$(f_n)$ is disjoint $\text{\rm w}^\ast$-null yet 
$|f_n|({\mathbb 1}_{\mathbb N})\equiv 2$ \cite[Ex.2.1]{CCJ}.

\begin{definition}\label{Schur property}
{\em A Banach lattice $E$ has:
\begin{enumerate}[a)]
\item 
the {\em positive Schur property} (or $E\in\text{\rm (PSP)}$)
if each positive \text{\rm w}-null sequence in $E$ is norm null (cf. \cite{Wnuk3});
\item 
the {\em dual positive Schur property} (or $E\in\text{\rm (DPSP)}$) if each positive 
\text{\rm w}$^\ast$-null sequence in $E'$ is norm null (cf. \cite[Def.3.3]{AEW});
\item 
the {\em dual disjoint Schur property} (or $E\in\text{\rm (DDSP)}$)
if each disjoint \text{\rm w}$^\ast$-null sequence in $E'$ is norm null 
(cf. \cite[Def.3.2)]{MEM} and \cite[Def.2.1.3~d)]{AEG3});
\item
the {\em property} (d) (or $E\in\text{\rm (d)}$) if, 
for every disjoint $\text{\rm w}^\ast$-null $(f_n)$ in $E'$,
the sequence $(|f_n|)$ is also $\text{\rm w}^\ast$-null (cf. \cite[Def.1]{El});
\item 
the {\em positive Grothendieck property} (or $E\in\text{\rm (PGP)}$) if each positive 
\text{\rm w}$^\ast$-null sequence in $E'$ is \text{\rm w}-null (cf. \cite[p.760]{Wnuk3}); 
\item 
the {\em disjoint Grothendieck property} (or $E\in\text{\rm (DGP)}$) 
if each disjoint \text{\rm w}$^\ast$-null sequence in $E'$ is \text{\rm w}-null 
(cf. \cite[Def.4.8)]{MFMA} and \cite[Def.2.1.3~e)]{AEG3}).
\item
the bi-{\em sequence property} (or $E\in\text{\rm (bi-sP)}$
if $f_n(x_n)\to 0$ for each \text{\rm w}$^\ast$-null 
$(f_n)$ in $E'_+$ and each disjoint $\text{\rm w}$-null $(x_n)$ in $E$
(cf. \cite[Def.3.1]{AEW}). 
\end{enumerate}}
\end{definition}
\noindent
By Assertion \ref{disj w-null is mod w-null}, $E\in\text{\rm (bi-sP)}$ iff
$f_n(x_n)\to 0$ for each \text{\rm w}$^\ast$-null 
$(f_n)$ in $E'_+$ and each disjoint $\text{\rm w}$-null $(x_n)$ in $E_+$.
The property \text{\rm (d)} is weaker than the sequential 
$\text{\rm w}^\ast$-continuity of the lattice operations, 
e.g. in $\ell^{\infty}$. It was proved in \cite[Prop.1.4]{Wnuk3}, 
where no name was assigned to the property~(d), that every 
Dedekind $\sigma$-complete $E$ has the property (d), 
and it was observed in \cite[Rem.1.5]{Wnuk3} 
that $\ell^{\infty}/c_0\in\text{\rm (d)}$
but $\ell^{\infty}/c_0$ is not Dedekind $\sigma$-complete.
The next proposition connects \text{\rm (DDSP)}, 
\text{\rm (DPSP)}, \text{\rm (DGP)}, 
and \text{\rm (d)}.

\begin{proposition}\label{DSP and (d)}
For a Banach lattice $E$, the following holds.
\begin{enumerate}[{\em i)}]
\item
$E\in\text{\rm (DDSP)} \ \Longrightarrow \ E\in\text{\rm (d)}$.
\item
$E\in\text{\rm (DGP)} \ \Longrightarrow \ E\in\text{\rm (d)}$.
\item
$E\in\text{\rm (DDSP)} \ \Longrightarrow \ E\in\text{\rm (DPSP)}$.
\item 
If $E\in\text{\rm (d)}$ then
$[E\in\text{\rm (DPSP)} \ \Longrightarrow \ E\in\text{\rm (DDSP)}]$.
\item 
If $E\in\text{\rm (d)}$ then
$[E\in\text{\rm (PGP)} \ \Longrightarrow \ E\in\text{\rm (DGP)}]$.
\end{enumerate}
\end{proposition}

\begin{proof}
i) 
Let $(f_n)$ be disjoint $\text{\rm w}^\ast$-null in $E'$.
Since $E\in\text{\rm (DDSP)}$ then $(f_n)$ is norm null.
By the norm-continuity of lattice operations, $(|f_n|)$ is norm null
and hence $\text{\rm w}^\ast$-null.

ii) 
Let $(f_n)$ be disjoint $\text{\rm w}^\ast$-null in $E'$.
Since $E\in\text{\rm (DGP)}$ then $(f_n)$ is \text{\rm w}-null in $E'$.
By Assertion \ref{disj w-null is mod w-null}, the sequence $(|f_n|)$ is \text{\rm w}-null
and hence $\text{\rm w}^\ast$-null.

iii) and iv) were proved in \cite[Thm.3.3]{MEM} (cf. also \cite[Prop.2.2.1]{AEG3}).

v) 
Let $(f_n)$ be disjoint $\text{\rm w}^\ast$-null in $E'$. Since $E\in\text{\rm (d)}$
then $(|f_n|)$ is also $\text{\rm w}^\ast$-null. Assuming $E\in\text{\rm (PGP)}$,
we obtain that $(|f_n|)$ is $\text{\rm w}$-null in $E'$. Let $g\in E''$. 
Then $|g(f_n)|\le|g|(|f_n|)\to 0$ for all $g\in E''$, and $(f_n)$
is \text{\rm w}-null.
\end{proof}
\noindent
The example mentioned after Assertion \ref{disj w-null is mod w-null}
shows that the property~(d) cannot be dropped in iv) and 
in v) of Proposition~\ref{DSP and (d)}. 

\subsection{}
By Proposition \ref{uniform convergence on A}~i),
$E\in\text{\rm (bi-sP)}$ iff each $\text{\rm w}^\ast$-null $(f_n)$
in $E'_+$ is uniformly null on each dis\-joint \text{\rm w}-null $(x_n)$
in $E_+$ (cf. \cite[Thm.4.2]{AEW}).
The next fact \cite[Prop.2.3]{Wnuk3} 
is essentially based on Assertion \ref{Dodds-Fremlin}(ii).

\begin{assertion}\label{Wnuk2013}
$E\in\text{\rm (DPSP)}$ iff every disjoint \text{\rm w}$^\ast$-null 
sequence in $E'_+$ is norm null.
\end{assertion}
\noindent
It follows $\text{\rm (DPSP)}\Longrightarrow \text{\rm (PGP)}\cap\text{\rm (DGP)}$. 
In \text{\rm AM}-spaces, \text{\rm (PGP)} $\Longleftrightarrow$ \text{\rm (DPSP)} 
by \cite[Prop.4.1]{Wnuk3}. 
\text{\rm (PSP)} $\Longleftrightarrow$ \text{\rm (SP)} in discrete Banach lattices 
\cite[p.19]{Wnuk3} and
\text{\rm (PGP)} $\Longleftrightarrow$ \text{\rm (GP)} in Banach lattices 
with the interpolation property by \cite[Thm.5.3.13]{Mey}. 
Recall that a subset $A$ of $E$ is {\em almost order bounded} if, 
for each $\varepsilon>0$, 
there is $x\in E_+$ with $A \subseteq [-x,x]+\varepsilon B_E$.

\begin{assertion}\label{Schur}
{\em (\cite[Cor.3.6.8]{Mey}, \cite[Thm.7]{Wnuk1})}
For a Banach lattice $E$, the following are equivalent.
\begin{enumerate}[\em (i)]
\item 
$E\in\text{\rm (PSP)}$.
\item 
Each disjoint $\text{\rm w}$-null sequence in $E$ is norm null.
\item 
Each disjoint $\text{\rm w}$-null sequence in $E_+$ is norm null.
\item 
Each disjoint sequence in the solid hull of every relatively 
$\text{\rm w}$-compact subset of $E$ is norm-null.
\item 
Almost order bounded subsets of $E$ coincide with 
relatively $\text{\rm w}$-compact subsets of $E$.
\end{enumerate}
\end{assertion}

\begin{definition}\label{semi-compact operator}
{\em An operator $T:X\to F$ is {\em semi-compact} 
if it carries bounded subsets of $X$ onto 
almost order bounded subsets of $F$. 
We denote by $\text{\rm semi-K}(X,F)$ the space 
of all semi-compact operators from $X$ to $F$.}
\end{definition}
\noindent
It is well known that 
$\text{\rm semi-K}(X,F)\subseteq\text{\rm W}(X,F)$ for every $X$,
whenever the norm in $F$ is o-continuous.
In view of Assertion \ref{Schur}(v), if $F\in\text{\rm (PSP)}$ then
$\text{\rm semi-K}(X,F)=\text{\rm W}(X,F)$ for every $X$.

By \cite[Def.1.1]{AAT}: a subset $A\subseteq E$ is called b-{\em bounded} 
if $i(A)$ is order bounded in $E''$, where $i:E\to E''$ is the natural embedding 
of $E$ into its bi-dual $E''$; and $E$ has b-{\em property}, if every b-bounded 
subset of $E$ is order bounded. The dual $E'$ has the b-property for each $E$.
An operator $T:E\to Y$ is said to be \text{\rm o-w}-compact
(resp. \text{\rm b-w}-compact) if $T$ carries order intervals
(resp. \text{\rm b}-bounded sets) of $E$ onto relatively 
\text{\rm w}-compact subsets of $Y$. Let $\text{\rm o-W}(E,Y)$ 
(resp. $\text{\rm b-W}(E,Y)$) be the space of all \text{\rm o-w}-compact
(resp. \text{\rm b-w}-compact) operators from $E$ to $Y$.  
Then $\text{\rm W}(E,Y)\subseteq\text{\rm b-W}(E,Y)\subseteq\text{\rm o-W}(E,Y)$ 
and all the inclusions are proper in general \cite{AAT}.

\subsection{}
The following definition was introduced in \cite[Def.0.1]{Mey0}.

\begin{definition}\label{LWC-subsets}
{\em 
A bounded subset $L$ of $F$ is called an \text{\rm LW}-subset if every 
disjoint sequence in $\text{\rm sol}(L)$ is norm null.}
\end{definition}
\noindent
\text{\rm LW}-sets are relatively 
\text{\rm w}-compact \cite[Prop.3.6.5.]{Mey}.
In view of Assertion~\ref{Schur}, relatively 
\text{\rm w}-compact subsets of $E$ are
\text{\rm LW}-sets iff $E\in\text{\rm (PSP)}$. 
The next characterization of LW-sets 
(see \cite[Prop.3.6.2]{Mey} and \cite[Lm.2.2]{BLM0}) is useful.
Notice that the equivalence (i)$\Longleftrightarrow$(iv) below is a 
direct consequence of Proposition~\ref{uniform convergence on A}
and Assertion~\ref{Burkinshaw--Dodds}.

\begin{assertion}\label{Meyer 3.6.2}
For a nonempty bounded subset $L$ of $E$, 
the following are equivalent.
\begin{enumerate}[\em (i)]
\item 
$L$ is an \text{\rm LW}-set.
\item 
Every disjoint bounded $(f_n)$ in $E'$ is uniformly null on $L$.
\item 
For every $\varepsilon>0$, there is 
$u_\varepsilon\in (E^a)_+$ such that
$L\subseteq[-u_\varepsilon,u_\varepsilon]+\varepsilon B_E$.
\item 
$f_n(x_n)\to 0$ for every disjoint bounded $(f_n)$ 
in $E'$ and every $(x_n)$ in $L$.
\end{enumerate}
\end{assertion}
\noindent
We shall use also the following two definitions.

\begin{definition}\label{Main LW operator}
{\em A continuous operator $T:X\to F$ is called$:$
\begin{enumerate}[a)]
\item  
$\text{\rm L}$-{\em weakly compact} ($T$ is an $\text{\rm LW}$-operator) 
if $T$ carries bounded subsets of $X$ onto $\text{\rm LW}$-subsets 
of $F$ \cite[Def.1.iii)]{Mey0};
\item 
{\em almost} $\text{\rm L}$-{\em weakly compact} 
($T$ is an $\text{\rm a-LW}$-operator) 
if $T$ carries relatively $\text{\rm w}$-compact subsets of $X$ onto 
$\text{\rm LW}$-subsets of $F$ 
\cite[Def.2.1]{BLM1}.
\end{enumerate}}
{\em A continuous operator $T:E\to F$ is called$:$
\begin{enumerate}[c)]
\item  
{\em order} $\text{\rm L}$-{\em weakly compact} 
($T$ is an $\text{\rm o-LW}$-operator) 
if $T$ carries order bounded subsets of $E$ onto 
$\text{\rm LW}$-subsets of $F$ \cite[Def.2.1]{BLM2}$;$
\end{enumerate}
\begin{enumerate}[d)]
\item  
$\text{\rm b-L}$-{\em weakly compact} ($T$ is a $\text{\rm b-LW}$-operator) 
if $T$ carries b-bounded subsets of $E$ onto $\text{\rm LW}$-subsets 
of $F$ \cite[Def.2.1]{BLM0}.
\end{enumerate}}
\end{definition}

\begin{definition}\label{Main MW operator}
{\em A continuous operator $T:E\to Y$ is called$:$
\begin{enumerate}[a)]
\item 
M-{\em weakly compact} ($T$ is an $\text{\rm MW}$-operator) if $\|Tx_n\|\to 0$ 
for every disjoint bounded $(x_n)$ in $E$ \cite[Def.1.iv)]{Mey0}$;$
\item 
{\em almost $\text{\rm M}$-weakly compact} 
($T$ is an $\text{\rm a-MW}$-operator)  
if $f_n(Tx_n)\to 0$ for every $\text{\rm w}$-convergent 
$(f_n)$ in $Y'$ and every disjoint 
bounded $(x_n)$ in $E$ \cite[Def.2.2]{BLM1}.
\end{enumerate}}
{\em A continuous operator $T:E\to F$ is called$:$
\begin{enumerate}[c)]
\item 
{\em order} \text{\rm M}-{\em weakly compact} 
($T$ is an \text{\rm o-MW}-operator) 
if $f_n(Tx_n)\to 0$ for every order bounded $(f_n)$ in $F'$ and every 
disjoint bounded $(x_n)$ in $E$ \cite[Def.2.2]{BLM0}.
\end{enumerate}}
\end{definition}

\begin{assertion}\label{Meyer-Nieberg}
{\em \cite[Satz.3]{Mey0}}
\begin{enumerate}[\em (i)]
\item 
$S'\in\text{\rm LW}(Y',E')\Longleftrightarrow S\in\text{\rm MW}(E,Y)$.
\item 
$T'\in\text{\rm MW}(F',X')\Longleftrightarrow T\in\text{\rm LW}(X,F)$.
\end{enumerate}
\end{assertion}

\begin{assertion}\label{Bouras--Lhaimer--Moussa}
{\em \cite[Thm.2.5]{BLM1}}.
\begin{enumerate}[\em (i)]
\item 
$S'\in\text{\rm a-LW}(Y',E')\Longleftrightarrow S\in\text{\rm a-MW}(E,Y)$.
\item 
$T'\in\text{\rm a-MW}(F',X')\Longrightarrow T\in\text{\rm a-LW}(X,F)$.
\end{enumerate}
\end{assertion}

\begin{assertion}\label{Bouras--Lhaimer--Moussa - order}
{\em (\cite[Thm.2.3]{BLM0} and \cite[Thm.2.3]{BLM2})}
\begin{enumerate}[\em (i)]
\item 
$S'\in\text{\rm o-LW}(F',E')\Longleftrightarrow S'\in\text{\rm b-LW}(F',E')
\Longleftrightarrow S\in\text{\rm o-MW}(E,F)$.
\item 
$T'\in\text{\rm o-MW}(F',E')\Longrightarrow T\in\text{\rm b-LW}(E,F)
\Longrightarrow T\in\text{\rm o-LW}(E,F)$.
\end{enumerate}
\end{assertion}
\noindent
In general, the implications in Assertion 
\ref{Bouras--Lhaimer--Moussa}(ii) and in 
Assertion \ref{Bouras--Lhaimer--Moussa - order}(ii) 
are proper (see \cite[Rem.2.1]{BLM1} and \cite[Rem.2.3]{BLM2}).
As an immediate consequence of Assertions \ref{Meyer-Nieberg}, 
\ref{Bouras--Lhaimer--Moussa}, 
and \ref{Bouras--Lhaimer--Moussa - order} 
we have the following (semi-) bi-duality.

\begin{assertion}\label{(semi-) bi-duality}
\begin{enumerate}[\em (i)]
\item 
$T''\in\text{\rm LW}(X'',F'')\Longleftrightarrow T\in\text{\rm LW}(X,F)$.
\item 
$S''\in\text{\rm MW}(E'',Y'')\Longleftrightarrow S\in\text{\rm MW}(E,Y)$.
\item 
$T''\in\text{\rm a-LW}(X'',F'')\Longrightarrow T\in\text{\rm a-LW}(X,F)$.
\item 
$S''\in\text{\rm a-MW}(E'',Y'')\Longrightarrow S\in\text{\rm a-MW}(E,Y)$.
\item 
$T''\in\text{\rm b-LW}(E'',F'')\Longleftrightarrow T''\in\text{\rm o-LW}(E'',F'')
\Longrightarrow T\in\text{\rm b-LW}(E,F)$.
\item 
$S''\in\text{\rm o-MW}(E'',F'')\Longrightarrow S'\in\text{\rm b-LW}(F',E')
\Longleftrightarrow S\in\text{\rm o-MW}(E,F)$.
\end{enumerate}
\end{assertion} 

\subsection{}
The range of an LW-operator is contained 
in a Banach lattice with o-continuous norm (cf. \cite[Thm.5.66]{AlBu}). 
The next proposition generalizes this fact to o-LW-operators.

\begin{proposition}\label{generalize to o-LW}
Let an operator $T\in\text{\rm o-LW}(E,F)$ be interval preserving. 
Then the norm closure $\bar{A}$ of the ideal $A$ generated by $T(E)$ 
in $F$ is a Banach lattice with \text{\rm o}-continuous norm. 
\end{proposition}

\begin{proof}
It suffices to show that every disjoint order bounded sequence 
in $A$ is norm null. So, let  $(x_n)$ be disjoint in $[0,x]$ for some $x\in A$. 
Pick $y_1,\dots, y_k$ in $E$ with $x\le\sum_{i=1}^k |Ty_i|$.
Using the Riesz decomposition property, 
we can write each $x_n$ as $x_n=\sum_{i=1}^k x_n^i$ with
$
  0\le x_n^i\le|Ty_i|\le T|y_i|\le T\left(\sum_{i=1}^k |y_i|\right)
$
for each $n$ and  $i=1,\dots, k$.
For each $i$, the sequence $(x^i_n)$ is disjoint 
in the \text{\rm LW}-set 
$T\left[0,\sum_{i=1}^k |y_i|\right]=\left[0,\sum_{i=1}^k T|y_i|\right]$,
and hence $\lim_n \|x^i_n\|=0$. Consequently $\lim_n\|x_n\|=0$.
\end{proof}

\begin{definition}\label{T has o-cont norm}
{\em (See \cite[p.328]{AlBu}).
Let $T\in{\cal L}_{ob}(E,F)$ where $F$ is Dedekind complete. 
Then $T$ is said to have {\em \text{\rm o}-continuous norm}, 
whenever $\|T_n\|\downarrow 0$. 
for every sequence $(T_n)$ of operators satisfying  
$|T|\ge T_n\downarrow 0$ in ${\cal L}_{ob}(E,F)$.}
\end{definition}

\begin{proposition}\label{a-LWo has o-cont norm}
Let the norms in $E'$ and $F$ be \text{\rm o}-continuous. 
Then each order bounded operator $T\in\text{\rm a-LW}(E,F)$ 
has \text{\rm o}-continuous norm.
\end{proposition}

\begin{proof}
By \cite[Thm.4]{EAS}, if $E'$ has o-continuous norm, 
then each order bounded a-LW-operator
$T: E\to F$ is MW. Hence, by \cite[Cor.3.6.14]{Mey}, 
$T$ is also LW. Then $T$ has o-continuous norm
by \cite[Thm.5.68]{AlBu}. 
\end{proof}

\subsection{} 
In Section 2, we introduce and investigate enveloping norms
of the regularly ${\cal P}$-operators between Banach lattices.
Section 3 is devoted to various modifications of limited operators,
which have being appeared in the last decade.
Section 4 is devoted to the almost Grothendieck operators.
introduced recently in \cite{GM}.
In Section 5, the almost Dunford--Pettis operators are studied.
For further unexplained terminology and notations, we refer to 
\cite{AlBu,AAT,AEG1,AEG2,AEW,BLM1,Diestel,DU,EAS,Emel,
BLM0,BLM2,LM1,Mey,Wnuk1,Wnuk3,Za}.

%%%%%%%%%%%%%%%%%%%%
\section{Enveloping norms on spaces of regularly ${\cal P}$-operators} 
%%%%%%%%%%%%%%%%%%%%

In this section, we continue the investigation of regularly ${\cal P}$-operators 
initiated in \cite{Emel,AEG2}. We introduce enveloping norms on spaces 
of regularly ${\cal P}$-operators.

\subsection{Regularly ${\cal P}$-operators.}
An operator $T:E\to F$ is called {\em regular} 
if $T=T_1-T_2$ for some $T_1,T_2\in{\cal L}_+(E,F)$. We denote by ${\cal L}_r(E,F)$ 
(resp. ${\cal L}_{ob}(E,F)$, ${\cal L}_{oc}(E,F)$) the ordered space of all regular 
(resp. order bounded, \text{\rm o}-continuous) operators 
in ${\cal L}(E,F)$.  The space ${\cal L}_r(E,F)$ need not to be a vector lattice,
but it is a Banach space under the {\em regular norm}
\begin{equation}\label{def of regular norm}
   \|T\|_r:=\inf\{\|S\|:\pm T\le S\in\text{\rm L}(E,F)\}
\end{equation}
by \cite[Prop.1.3.6]{Mey}. Furthermore, for every $T\in{\cal L}_r(E,F)$,
\begin{equation}\label{regular norm 1}
   \|T\|_r=\inf\{\|S\|: S\in\text{\rm L}(E,F), |Tx|\le S|x|\ \forall x\in E\}\ge\|T\|.
\end{equation} 
If $F$ is Dedekind complete, then $({\cal L}_r(E,F),\|\cdot\|_r)$ is a Banach lattice
such that $\|T\|_r=\|~|T|~\|$ for every $T\in{\cal L}_r(E,F)$ \cite[Prop.1.3.6]{Mey}.
The following definition was introduced in \cite[Def.2]{Emel} 
(cf. also \cite[Def.1.5.1]{AEG3}). 

\begin{definition}\label{rP-operators}{\em
Let ${\cal P}\subseteq\text{\rm L}(E,F)$. An operator $T:E\to F$ is called 
a {\it regularly} ${\cal P}$-{\it operator}
(shortly, an r-${\cal P}$-{\it operator}), if $T=T_1-T_2$ 
with $T_1,T_2\in{\cal P}\cap\text{\rm L}_+(E,F)$.}
\end{definition}
\noindent
Given ${\cal P}\subseteq\text{\rm L}(E,F)$; we denote by: 
\begin{enumerate}[]
\item 
${\cal P}(E,F):= {\cal P}$ the set of all ${\cal P}$-{\em operators} 
in $\text{\rm L}(E,F)$;  
\item 
${\cal P}_r(E,F)$ the set of all regular operators in ${\cal P}(E,F)$;  
\item 
$\text{\rm r-}{\cal P}(E,F)$ the set of all regularly 
${\cal P}$-operators in $\text{\rm L}(E,F)$.
\end{enumerate}

\begin{assertion}\label{prop elem}
{\rm (\cite[Prop.1.5.2]{AEG3})}
Let ${\cal P}\subseteq\text{\rm L}(E,F)$, ${\cal P}\pm{\cal P}\subseteq{\cal P}\ne\emptyset$, 
and $T\in\text{\rm L}(E,F)$. Then the following holds.
\begin{enumerate}[\em (i)]
\item 
$T$ is an {\em r-}${\cal P}$-operator iff $T$ is a ${\cal P}$-dominated ${\cal P}$-operator.  
\item 
Suppose ${\cal P}$-operators satisfy the domination property and the modulus $|T|$ exists 
in $\text{\rm L}(E,F)$. Then $T$ is an \text{\rm r}-${\cal P}$-operator iff  $|T|\in{\cal P}$.
\end{enumerate}
\end{assertion}

The following fact was established in \cite[Prop.1.5.3]{AEG3}.

\begin{assertion}\label{vect lat}
Let $F$ be Dedekind complete, and let ${\cal P}$ be a subspace in $\text{\rm L}(E,F)$,
satisfiying the domination property. Then $\text{\rm r-}{\cal P}(E,F)$
is a Dedekind complete vector lattice.
\end{assertion}

\subsection{Enveloping norms.}
It is natural to replace $\text{\rm L}(E,F)$ in the definition (\ref{def of regular norm}) 
of the regular norm by an arbitrary nonempty ${\cal P}\subseteq\text{\rm L}(E,F)$ as follows:
\begin{equation}\label{enveloping norm}
   \|T\|_{\text{\rm r-}{\cal P}}:=\inf\{\|S\|:\pm T\le S\in{\cal P}\} \ \ 
   \ \ (T\in\text{\rm r-}{\cal P}(E,F)).
\end{equation}

\begin{lemma}\label{Enveloping P-norm}
For a vector subspace ${\cal P}$ of $\text{\rm L}(E,F)$, the formula 
$(\ref{enveloping norm})$ defines a norm on $\text{\rm r-}{\cal P}(E,F)$,
called the {\rm enveloping norm}. Moreover,
\begin{equation}\label{Enveloping P-norm 2}
   \|T\|_{\text{\rm r-}{\cal P}}=
   \inf\{\|S\|: S\in{\cal P}\ \&\ (\forall x\in E)\ |Tx|\le S|x|\}
   \ \ \ (T\in\text{\rm r-}{\cal P}(E,F)).
\end{equation}
If ${\cal P}_1$ is a vector subspace of ${\cal P}$ then
\begin{equation}\label{Enveloping P-norm 1}
   \|T\|_{\text{\rm r-}{\cal P}_1}\ge\|T\|_{\text{\rm r-}{\cal P}}\ge\|T\|_r\ge\|T\| 
   \ \ \ \ \ (\forall\ T\in\text{\rm r-}{\cal P}_1(E,F)).
\end{equation}
\end{lemma}

\begin{proof}
Only the triangle inequality for $\|\cdot\|_{\text{\rm r-}{\cal P}}$ and the 
formula (\ref{Enveloping P-norm 2}) require some explanations.

(A) Let $T_1,T_2\in\text{\rm r-}{\cal P}(E,F)$ and $\varepsilon>0$. Pick
$S_1,S_2\in{\cal P}$ with $\pm T_1\le S_1$, $\pm T_2\le S_2$,
$\|S_1\|\le\|T_1\|_{\text{\rm r-}{\cal P}}+\varepsilon$, and 
$\|S_2\|\le\|T_2\|_{\text{\rm r-}{\cal P}}+\varepsilon$.
Then $\pm(T_1+T_2)\le S_1+S_2\in{\cal P}$, and
$
    \|T_1+T_2\|_{\text{\rm r-}{\cal P}}\le\|S_1+S_2\|\le\|S_1\|+\|S_2\|\le
    \|T_1\|_{\text{\rm r-}{\cal P}}+\|T_2\|_{\text{\rm r-}{\cal P}}+2\varepsilon.
$
Since $\varepsilon>0$ is arbitrary, 
$\|T_1+T_2\|_{\text{\rm r-}{\cal P}}\le
\|T_1\|_{\text{\rm r-}{\cal P}}+\|T_2\|_{\text{\rm r-}{\cal P}}$.

(B) Denote the right side of (\ref{Enveloping P-norm 2}) by $R(T)$. 
If $\pm T\le S\in{\cal P}$ then 
$$
   \pm Tx=\pm(T(x_+)-T(x_-))=\pm T(x_+)\mp T(x_-)\le S(x_+)+S(x_-)=S|x|
$$
for all $x\in E$. Then $|Tx|\le S|x|$ for all $x\in E$, and hence 
$\|T\|_{\text{\rm r-}{\cal P}}\ge R(T)$.

If $S\in{\cal P}$ satisfies $|Tx|\le S|x|$ for all $x\in E$ then,
for all $y\in E_+$, $|Ty|\le Sy$ and consequently $\pm Ty\le Sy$.
Therefore $\pm T\le S$, and hence $R(T)\ge\|T\|_{\text{\rm r-}{\cal P}}$.
\end{proof}

\begin{corollary}\label{reg op banach lattice}
Let $F$ be Dedekind complete and let ${\cal P}$ 
be a closed in the operator norm subspace of $\text{\rm L}(E,F)$
satisfying the domination property. Then $\text{\rm r-}{\cal P}(E,F)$
is a Dedekind complete Banach lattice under the enveloping norm.
\end{corollary}

\begin{proof}
By Assertion~\ref{vect lat}, $\text{\rm r-}{\cal P}(E,F)$ is a Dedekind 
complete vector lattice. Theorem~\ref{P-norm} implies that 
$\text{\rm r-}{\cal P}(E,F)$ is a Banach space
under the enveloping norm.

Assertion~\ref{vect lat} implies  
$\|T\|_{\text{\rm r-}{\cal P}}=\|~|T|~\|=\|~|T|~\|_{\text{\rm r-}{\cal P}}$
for all $T\in\text{\rm r-}{\cal P}(E,F)$ and hence   
$(\text{\rm r-}{\cal P}(E,F),\|\cdot\|_{\text{\rm r-}{\cal P}})$ 
is a Banach lattice.
\end{proof}

\subsection{Completeness of the enveloping norms.}
The following theorem can be considered as an extension of \cite[Prop.1.3.6]{Mey}, 
\cite[Lm.1]{Emel} (see also \cite[Prop.2.2]{CW97} and 
\cite[Thm.2.3]{Cheng} for particular cases). Its proof is a rather 
straightforward modification of the proof of \cite[Prop.1.3.6]{Mey}.

\begin{theorem}\label{P-norm}
Let ${\cal P}$ be a subspace of $\text{\rm L}(E,F)$
closed in the operator norm.
Then $\text{\rm r-}{\cal P}(E,F)$ is a Banach space under the
enveloping norm.
\end{theorem}

\begin{proof}
Let $(T_n)$ be 
$\|\cdot\|_{\text{\rm r-}{\cal P}}$-Cauchy in $\text{\rm r-}{\cal P}(E,F)$,
say $T_n=P_n-R_n$ for $P_n,R_n\in{\cal P}\cap{\cal L}_+(E,F)$.
WLOG, we can assume that $\|T_{n+1}-T_n\|_{\text{\rm r-}{\cal P}}<2^{-n}$ for all
$n\in\mathbb{N}$. Since $\|\cdot\|_{\text{\rm r-}{\cal P}}\ge\|\cdot\|$, there exists some
$T\in\text{\rm L}(E,F)$ with $\|T-T_n\|\to 0$.
We obtain $T\in{\cal P}$ because $T_n\in{\cal P}$ and ${\cal P}$ is closed 
in the operator norm. Pick $S_n\in{\cal P}$ with $\|S_n\|<2^{-n}$ and
$\pm(T_{n+1}-T_n)\le S_n$. By (\ref{Enveloping P-norm 2}), $|(T_{n+1}-T_n)x|\le S_n|x|$ for 
all $x\in E$ and $n\in\mathbb{N}$. Since ${\cal P}$ is closed in the operator norm,
$Q_n:=\|\cdot\|\text{\rm -}\sum\limits_{k=n}^\infty S_k\in{\cal P}$ for each $n\in\mathbb{N}$.
Since
$$
   |(T-T_n)x|=\lim\limits_{k\to\infty}|(T_k-T_n)x|\le
   \sum\limits_{k=n}^\infty|(T_{k+1}-T_n)x|\le Q_n|x| \ \ \ \ \ (\forall x\in E), 
$$
then $\pm(T-T_n)\le Q_n$ for all $n\in\mathbb{N}$. 
Thus $-Q_n\le(T-T_n)\le Q_n$ and hence $0\le(T-T_n)+Q_n$
for all $n\in\mathbb{N}$. In particular, 
$$
   T=[(T-T_n)+Q_n]+[T_n-Q_n]=
   [(T-T_n)+Q_n+P_n]-[R_n+Q_n]\in\text{\rm r-}{\cal P}(E,F),
$$
and hence $(T-T_n)\in\text{\rm r-}{\cal P}(E,F)$ for all $n\in\mathbb{N}$.
Now, $\|T-T_n\|_{\text{\rm r-}{\cal P}}\le\|Q_n\|<2^{1-n}$ implies
$(T_n)\stackrel{\|\cdot\|_{\text{\rm r-}{\cal P}}}{\to}T$.
\end{proof}
\noindent
In general, $\text{\rm r-}{\cal P}(E,F)\subsetneqq\text{\rm r-}{\overline{\cal P}}(E,F)$,
where ${\overline{\cal P}}$ is the norm-closure of ${\cal P}$ in $\text{\rm L}(E,F)$.
From the other hand, for ${\cal P}:={\cal L}_{ob}(E,F)$ (which  
is almost never closed in $\text{\rm L}(E,F)$ in the operator norm),
$\text{\rm r-}{\cal P}(E,F)=\text{\rm r-}{\overline{\cal P}}(E,F)={\cal L}_r(E,F)$.
The enveloping norms on these three spaces agree with the regular norm,
which makes $\text{\rm r-}{\cal L}_{ob}(E,F)$ a Banach space 
by \cite[Prop.1.3.6]{Mey}.
The following proposition coupled with Example~\ref{P=o-cont}
shows that the enveloping norm on $\text{\rm r-}{\cal P}(E,F)$ can be 
complete even if 
$\text{\rm r-}{\cal P}(E,F)\ne\text{\rm r-}{\overline{\cal P}}(E,F)$.

\begin{proposition}\label{o-cont not so bad}
Let the norm in $F$ be \text{\rm o}-continuous. Then
$\text{\rm r-}{\cal L}_{oc}(E,F)$ is a Banach space 
under the enveloping norm.
\end{proposition}

\begin{proof}
Let $(T_n)$ be a Cauchy sequence in $\text{\rm r-}{\cal L}_{oc}(E,F)$ in 
the enveloping norm. WLOG, we can assume that 
$\|T_{n+1}-T_n\|_{\text{\rm r-}{\cal L}_{oc}(E,F)}<2^{-n}$ for all
$n\in\mathbb{N}$. Let $T\in\text{\rm L}(E,F)$ satisfy
$\|T-T_n\|\to 0$. Pick $S_n\in{\cal L}_{oc}(E,F)$ with $\|S_n\|<2^{-n}$ and
$\pm(T_{n+1}-T_n)\le S_n$. First, we claim
$Q_n:=\|\cdot\|\text{\rm -}\sum\limits_{k=n}^\infty S_k\in{\cal L}_{oc}(E,F)$ 
for all $n\in\mathbb{N}$. To prove the claim, it sufficies to show that
$Q_1\in{\cal L}_{oc}(E,F)$. So, let $x_\alpha\downarrow 0$ in $E$.
Passing to a tail we can assume that $\|x_\alpha\|\le M\in\mathbb{R}$
for all $\alpha$. Since $Q_1\ge 0$ then $Q_1x_\alpha\downarrow\ge 0$ and hence
in order to show that $Q_1x_\alpha\downarrow 0$ it is enough to
prove that $\|Q_1x_\alpha\|\to 0$.
Let $\varepsilon>0$. Fix an $m\in\mathbb{N}$ with
$\|Q_{m+1}\|\le\varepsilon$. Since the positive 
operators $S_1,...,S_m$ are all \text{\rm o}-continuous, 
and since the norm in $F$ is \text{\rm o}-continuous,
there exists an $\alpha_1$ such that 
$\sum\limits_{k=1}^m\|S_kx_\alpha\|\le\varepsilon$
for all $\alpha\ge\alpha_1$.  Since $\varepsilon>0$ is arbitrary,
it follows from 
$$
   \|Q_1x_\alpha\|\le\|\sum\limits_{k=1}^mS_kx_\alpha\|+
   \|Q_{m+1}x_\alpha\|\le\varepsilon+M\|Q_{m+1}\|\le 
   2\varepsilon \ \ \ \ (\forall\alpha\ge\alpha_1)
$$
that $\|Q_1x_\alpha\|\to 0$, which proves our claim that 
$Q_n\in{\cal L}_{oc}(E,F)$
for all $n\in\mathbb{N}$.

Since $\pm(T_{n+1}-T_n)\le S_n$ then
by formula (\ref{Enveloping P-norm 2}),
$$
   |(T-T_n)x|=\lim\limits_{k\to\infty}|(T_k-T_n)x|\le
   \sum\limits_{k=n}^\infty |(T_{k+1}-T_n)x|\le
   \sum\limits_{k=n}^\infty S_n|x|=Q_n|x|  
$$
for all $x\in E$. In particular, $|T-T_1|\le Q_1\in{\cal L}_{oc}(E,F)$ and,
since ${\cal L}_{oc}(E,F)$ is an order ideal in ${\cal L}_{r}(E,F)$, then
$(T-T_1)\in{\cal L}_{oc}(E,F)$. 
Since $T_1\in{\cal L}_{oc}(E,F)$, it follows
$T\in{\cal L}_{oc}(E,F)$.
Now, $\|T-T_n\|_{\text{\rm r-}{\cal L}_{oc}(E,F)}\le\|Q_n\|<2^{1-n}$ implies
$(T_n)\stackrel{\|\cdot\|_{\text{\rm r-}{\cal L}_{oc}(E,F)}}{\longrightarrow}T$.
\end{proof}

\begin{example}\label{P=o-cont}
Consider the modification of Krengel's example {\rm \cite[Ex.5.6]{AlBu}}
with $\alpha_n=2^{-\frac{n}{3}}$. The sequence $(K_n)$ in 
${\cal L}_{oc}((\oplus_{n=1}^{\infty}\ell^2_{2^n})_0)$, defined by
$$
   K_nx:=(\alpha_1T_1x_1,\alpha_2T_2x_2,\dots\alpha_nT_nx_n,0,0,\dots),
$$
converges in the operator norm to 
the operator $K\in\text{\rm L}(E,F)$, defined by
$$
   Kx:=(\alpha_1T_1x_1,\alpha_2T_2x_2,\dots\alpha_nT_nx_n,\dots).
$$
Notice that $(\oplus_{n=1}^{\infty}\ell^2_{2^n})_0$ is a Banach lattice
with \text{\rm o}-continuous norm.
Since $|K|$ does not exist, $K$ is not order bounded and hence
$K\not\in{\cal L}_{oc}((\oplus_{n=1}^{\infty}\ell^2_{2^n})_0)$.
\end{example}
\noindent
In the table below, we list several already known results on enveloping norms.
Other applications of Theorem \ref{P-norm} and Assertion \ref{prop elem}
are included in the following sections.

\vspace{3mm}
\noindent
{\tiny
\begin{tabular}{|p{25mm}|p{35mm}|p{3cm}|p{3cm}|}
\hline
$\text{\rm r-}{\cal P}(E,F)$ & 
is complete under the enveloping norm & 
consists of dominated operators & 
is a Banach lattice algebra when $E$ is Dedekind complete and $F=E$ \\
\hline
\text{\rm r-}compact & 
by \cite[Prop.2.2]{CW97} & 
+/- & 
+/- \\
\hline
\text{\rm r-b-AM-}compact & 
by \cite[Thm.2.3]{Cheng} & 
+/- & 
+/- \\
\hline
\text{\rm r-} ($\text{\rm r-}\sigma$) Lebesgue & 
cf. \cite[Thm.1]{Emel} & 
an easy exercise & 
\cite[Thm.1]{Emel}  \\
\hline
\text{\rm r-quasi-KB} & 
cf. \cite[Thm.2]{Emel} & 
by \cite[Thm.2.6]{AEG2} & 
by \cite[Thm.2]{Emel}  \\
\hline
\text{\rm r-a-LW} and \text{\rm r-a-MW} & 
\cite[Prop.2.1]{BLM1}, \cite[Thm.3.2.1]{AEG3} & 
by \cite[Thms.1, 2]{AkGo} & 
\cite[Thms.3.2.7, 3.3.9]{AEG3} \\
\hline
\text{\rm r-o-LW} and \text{\rm r-o-MW} & 
follows from Theorem \ref{P-norm} and \cite[Cor.2.3]{BLM2} & 
\cite[Cor.2.3]{BLM2} & 
by Theorem \ref{P-norm} and Assertion \ref{prop elem}\\
\hline
\end{tabular}
}

%%%%%%%%%%%%%%%%%%%%
\section{Almost limited operators.} 
%%%%%%%%%%%%%%%%%%%%

In this section, we investigate several modifications of limited operators,
introduced recently in \cite{CCJ,El,EMM,FKM,Ghenciu,MF,KFMM,MFMA}.

\subsection{Almost limited, almost Dunford--Pettis, and almost L-sets.}
Standard application of the Banach lattice structure to the definitions of limited,
Dunford--Pettis, and \text{\rm L}-sets gives the following.  

\begin{definition}\label{alim and a-DP}
{\em  A bounded subset $A\subseteq E$ is called:
\begin{enumerate}[a)]
\item 
{\em almost limited} (an \text{\rm a}-{\em limited set}) if every disjoint 
\text{\rm w}$^\ast$-null sequence in $E'$ is uniformly null 
on $A$ \cite[Def.2.3]{CCJ};
\item 
{\em almost Dunford--Pettis} (an \text{\rm a-DP}-{\em set}) 
if every disjoint \text{\rm w}-null sequence in $E'$ 
is uniformly null on $A$ \cite{Bou}.
\end{enumerate}}
{\em  A bounded subset $B\subseteq E'$ is called:
\begin{enumerate}[c)]
\item 
an {\em almost} \text{\rm L}-{\em set} (shortly, an \text{\rm a-L}-{\em set})
if every disjoint \text{\rm w}-null sequence 
$(x_n)$ in $E$ is uniformly null on $B$ (cf. \cite{AqBo,Emma}). 
\end{enumerate}}
\end{definition}
\noindent
It follows that every \text{\rm a}-limited set is 
\text{\rm a-DP}, and every \text{\rm a-DP}-set is bounded.
Since \text{\rm w}$^\ast$-null sequences in $E'$ are bounded,
Assertion \ref{Meyer 3.6.2} implies that each \text{\rm LW}-set 
is \text{\rm a}-limited. The next fact is an immediate consequence
of Definition \ref{Main Schur property}~b), 
Definition \ref{Schur property}~f), and 
Definition \ref{alim and a-DP}.

\begin{assertion}\label{GP vs a-DP}
The following holds.
\begin{enumerate}[{\em (i)}] 
\item
Let $A\subseteq X\in\text{\rm (GP)}$. Then
$A$ is a \text{\rm DP-}set iff $A$ is limited.
\item
Let $A\subseteq E\in\text{\rm (DGP)}$. Then
$A$ is an \text{\rm a-DP-}set iff $A$ is \text{\rm a-}limited.
\end{enumerate}
\end{assertion}
\noindent
By \cite[Thm.2.6]{CCJ}, the norm in $E$ 
is $\text{\rm o}$-continuous iff each \text{\rm a}-limited subset 
of $E$ is an \text{\rm LW}-set.
The next fact (cf. Assertion \ref{limited and DP sets})
follows directly from Definition \ref{alim and a-DP}
and Proposition \ref{uniform convergence on A}.

\begin{assertion}\label{a-limited and a-DP sets}
Let $A$ be a bounded subset of $E$. Then
\begin{enumerate}[{\em (i)}]
\item
$A$ is an \text{\rm a}-limited set iff $f_n(a_n)\to 0$ for every disjoint 
\text{\rm w}$^\ast$-null $(f_n)$ in $E'$ and every $(a_n)$ in $A$.
\item
$A$ is an \text{\rm a-DP}-set iff $f_n(a_n)\to 0$ for every disjoint 
\text{\rm w}-null $(f_n)$ in $E'$ and every $(a_n)$ in $A$ 
{\em \cite[Prop.2.1]{Bou}}.
\end{enumerate}
A bounded subset $B$ of $E'$ is an \text{\rm a-L}-{\em set}
iff $b_n(x_n)\to 0$ for every $(b_n)$ in $B$ and every 
disjoint \text{\rm w}-null $(x_n)$ in $E$.
\end{assertion}

\begin{proposition}\label{b-bounded set is a-DP}
Every \text{\rm b}-bounded subset $A$ of $E$ is an \text{\rm a-DP}-set.
\end{proposition}

\begin{proof}
Let $(x'_n)$ be a disjoint \text{\rm w}-null sequence in $E'$ and $A\subseteq E$ 
be \text{\rm b}-bounded. 
As the \text{\rm w}-convergent sequence $(x'_n)$ is bounded, 
without loss of generality, 
we may assume $(x'_n)\subset B_{E'}$. 
There exists $f\in E''_+$ such that $A\subset [-f,f]$. 
Given $\varepsilon>0$, 
there is $n_\varepsilon\in{\mathbb N}$ with $|f(x'_n)|<\varepsilon$ 
for all $n \ge n_\varepsilon$. For every $a\in A$:
$$
   |x'_n(a)|\le |x'_n|(|a|)\le|x'_n|(f)= f(|x'_n|)<\varepsilon \ \ \ (\forall n\ge n_\varepsilon).
$$
Therefore $\sup\limits_{a\in A}|x'_n(a)|<\varepsilon$ if $n \ge n_\varepsilon$, 
and hence $A$ is an
\text{\rm a-DP}-set.
\end{proof}
\noindent 
Clearly, $[-f,f]$ is an \text{\rm a-L}-set for each $f\in E'_+$,
and $B_{\ell^\infty}$ is an \text{\rm a-L}-set in $\ell^\infty$, 
but $B_{\ell^1}$ is not an \text{\rm a-L}-set in $\ell^1$.
In fact, each order bounded subset of $E'$ is an \text{\rm a-L}-set
by Assertion~\ref{disj w-null is mod w-null}.
\text{\rm L}-subsets of $E$ coincide with \text{\rm a-L}-subsets iff lattice operations 
in $E$ are sequentially \text{\rm w}-continuous \cite[Thm.4.1]{AqBo}.
Observe that $[-y, y]$ is an \text{\rm a-L}-set in $E''$ 
for each $y\in E_+''$. Indeed, it follows from the inequality 
$|f(x_n)|\le|f|(|x_n|)\le y(|x_n|)$ for all $x_n\in E'$, $y\in E''_+$, 
and all $f\in [-y,y]$. Since each \text{\rm b}-bounded subset of $E$ 
can be considered as a subset of $[-y,y]$ for some $y\in E_+''$, 
it follows that every \text{\rm b}-bounded subset of $E$ 
is an \text{\rm a-L}-subset of $E''$.

\subsection{}
As every Dedekind $\sigma$-complete $E$ has 
the property (d) \cite[Prop.1.4]{Wnuk3},
the following proposition slightly generalizes \cite[Thm.2.5]{CCJ}
(e.g. $\ell^{\infty}/c_0\in\text{\rm (d)}$ yet $\ell^{\infty}/c_0$ is 
not Dedekind $\sigma$-complete \cite[Rem.1.5]{Wnuk3}). We skip 
the proof, as it is essentially 
the same with the proof in \cite[Thm.2.5]{CCJ}.

\begin{proposition}\label{a-limited and a-DP sets}
Let $E\in\text{\rm (d)}$ and let $A$ be a bounded solid subset of $E$.
Then the following statements are equivalent.
\begin{enumerate}[\em i)]
\item
$A$ is \text{\rm a}-limited.
\item
$f_n(a_n)\to 0$ for every disjoint \text{\rm w}$^\ast$-null $(f_n)$ in $E'$ 
and every disjoint $(a_n)$ in $A$.
\item
$f_n(a_n)\to 0$ for every disjoint \text{\rm w}$^\ast$-null $(f_n)$ in $E'_+$ 
and every disjoint $(a_n)$ in $A\cap E_+$.
\end{enumerate}
\end{proposition}
\noindent
We do not know whether or not Proposition \ref{a-limited and a-DP sets} 
still holds true without the assumption $E\in\text{\rm (d)}$.

\subsection{Modifications of limited operators.}
We list the main definitions.

\begin{definition}\label{Main def of limited operators} 
{\em A continuous operator 
\begin{enumerate}[a)]
\item 
$T:X\to Y$ is called {\em limited} if $T(B_X)$ is limited \cite{Diestel};
i.e., $T'$ takes $\text{\rm w}^\ast$-null sequences of $Y'$ to norm null 
sequences of $X'$.
\item 
$T:X\to F$ is called {\em almost limited} if $T(B_X)$ is \text{\rm a}-limited 
\cite{EMM};
i.e., $T'$ takes disjoint $\text{\rm w}^\ast$-null sequences of $F'$ to 
norm null sequences of $X'$. 
\item 
$T:E\to Y$ is called {\em $\text{\rm o}$-limited}, if $T[0,x]$ is limited for all 
$x\in E_+$ \cite{KFMM};
i.e., for every $\text{\rm w}^\ast$-null sequence $(f_n)$ of $Y'$, 
$(T'f_n)$ is uniformly null on each order interval $[0,x]$ of $E_+$. 
\item 
$T:E\to F$ is called {\em almost $\text{\rm o}$-limited}, if $T[0,x]$ is \text{\rm a}-limited
for all $x\in E_+$;
i.e., for every disjoint $\text{\rm w}^\ast$-null sequence $(f_n)$ of $F'$,
$(T'f_n)$ is uniformly null on each $[0,x]\subseteq E_+$.
\item $T:E\to Y$ is called {\em \text{\rm b}-limited}, if $T(A)$ is limited for each 
\text{\rm b}-bounded subset $A$ of $E$.
\item $T:E\to F$ is called {\em almost \text{\rm b}-limited}, 
if $T(A)$ is \text{\rm a}-limited for each \text{\rm b}-bounded 
subset $A$ of $E$.
\end{enumerate}}
\end{definition}
\noindent
By $\text{\rm Lm}(X,Y)$, $\text{\rm a-Lm}(X,F)$, 
$\text{\rm o-Lm}(E,Y)$, 
$\text{\rm a-o-Lm}(E,F)$, $\text{\rm b-Lm}(E,Y)$, 
and $\text{\rm a-b-Lm}(E,F)$ we denote the spaces of operators 
in Definition~\ref{Main def of limited operators}.
Since compact subsets of $Y$ are limited and limited sets of $Y$ are bounded,
\begin{equation}\label{1b}
  {\text{\rm K}}(X,Y)\subseteq\text{\rm Lm}(X,Y)\subseteq\text{\rm L}(X,Y).
\end{equation}
In the case of operators from $E$ to $F$, we have 
the following obvious inclusions:
\begin{equation}\label{2b}
  \text{\rm Lm}(E,F)\subseteq\text{\rm a-Lm}(E,F)\cap\text{\rm b-Lm}(E,F);
\end{equation}
\begin{equation}\label{3b}
  \text{\rm b-Lm}(E,F)\subseteq\text{\rm o-Lm}(E,F)\subseteq\text{\rm a-o-Lm}(E,F);
\end{equation}
\begin{equation}\label{4b}
  \text{\rm b-Lm}(E,F)\subseteq\text{\rm a-b-Lm}(E,F)\cap\text{\rm a-o-Lm}(E,F).
\end{equation}

\begin{example}\label{limited examples}
{\em All inclusions in \eqref{2b}--\eqref{4b} are proper.
\begin{enumerate}[a)]
\item
Consider the identity operator $I_{\ell^\infty}$.
Since $B_{\ell^\infty}$ is an \text{\rm a}-limited but not limited set, then 
\begin{equation}\label{5b}
  I_{\ell^\infty}\in(\text{\rm a-Lm}(\ell^\infty)\cap\text{\rm b-Lm}(\ell^\infty))
  \setminus\text{\rm Lm}(\ell^\infty).
\end{equation}
\item
Since order bounded subsets of $L^1[0,1]$ are limited \cite[Rem.2.4.(2)]{CCJ} 
and coincide with b-bounded subsets, and $B_{L^1[0,1]}$ is \text{\rm a}-limited 
but not limited, then
\begin{equation}\label{6b}
  I_{L^1[0,1]}\in\text{\rm b-Lm}({L^1[0,1]})\setminus\text{\rm Lm}({L^1[0,1]}).
\end{equation}
\item
Since relatively compact subsets of $\ell^1$ coincide with limited subsets 
and with \text{\rm a}-limited subsets, then 
\begin{equation}\label{7b}
  I_{\ell^1}\in(\text{\rm a-b-Lm}(\ell^1)\cap
  \text{\rm o-Lm}(\ell^1))\setminus\text{\rm Lm}(\ell^1).
\end{equation}
\end{enumerate}}
\end{example}

\begin{assertion}\label{o-limited MW T is limited} 
{\em (\cite[Thm.7]{FKM})}
$\text{\rm o-Lm}(E,F)\cap\text{\rm MW}(E,F)\subseteq\text{\rm Lm}(E,F)$.
\end{assertion}

\begin{proposition}\label{o-LW MW is alimited}
$
  \text{\rm o-LW}(E,F)\cap\text{\rm MW}(E,F)\subseteq\text{\rm a-Lm}(E,F).
$
\end{proposition}

\begin{proof}
Let $T\in\text{\rm o-LW}(E,F)\cap\text{\rm MW}(E,F)$ and let $(f_n)$ be 
disjoint  $\text{\rm w}^\ast$-null in $B_F$. We claim $\|T'f_n\|\to 0$.
By Assertion~\ref{Dodds-Fremlin}(ii), it suffices to show that: 

1) 
$(|T'f_n|)$ is $\text{\rm w}^\ast$-null; and 

2) 
$|T'f_n(x_n)|\to 0$ for each disjoint bounded 
sequence $(x_n)$ in $E_+$.\\
Since $T$ is o-LW, $|T'f_n|\stackrel{\text{\rm w}^\ast}{\to} 0$ 
by \cite[Thm.2.1]{BLM2}, which proofs 1).
Let $(x_n)$ be a disjoint bounded sequence in $E_+$. 
Since $T$ is MW, $\|Tx_n\|\to 0$. In view of 
$$
  |T'f_n(x_n)|=|f_n(Tx_n)|\le\|f_n\|\cdot \|Tx_n\|\le \|Tx_n\|,
$$
$|T'f_n(x_n)|\to 0$, which proofs 2).\\
Therefore $\|T'f_n\|\to 0$ and hence $T\in\text{\rm a-Lm}(E,F)$.
\end{proof}
\noindent
It is proved in \cite{MF} that the domination property is satisfied for: 
\begin{enumerate}[(i)]
\item 
$\text{\rm o}$-limited operators from every $E$ to 
every Dedekind $\sigma$-complete $F$;
\item
limited operators from every $E$ such that the norm in $E'$ 
is $\text{\rm o}$-continuous 
to every Dedekind $\sigma$-complete $F$. 
\end{enumerate}
\noindent
An operator $T:E\to X$ is said to be AM-{\em compact} 
if $T[0,x]$ is relatively 
compact for each $x\in E_+$.  
Thus every AM-compact operator is o-limited. 
An operator $T:E\to X$ is said to be b-AM-{\em compact} if it carries 
b-bounded subsets of $E$ to a relatively compact subsets of $X$ \cite{AM}. 
As every b-AM-compact operator is AM-compact, 
each b-AM-compact operator is o-limited. 

\subsection{Some conditions under which operators are (almost) limited.} 

\begin{proposition}\label{semi-compact vs limited}
For any Banach space $X$$:$
\begin{enumerate}[\em i)]
\item
If $F\in\text{\rm (d)}$, 
then $\text{\rm se\-mi-K}(X,F)\subseteq\text{\rm a-Lm}(X,F)$.
\item
If the lattice operations in $F$ are 
sequentially $\text{\rm w}^\ast$-continuous, then\\ 
$\text{\rm se\-mi-K}(X,F)\subseteq \text{\rm Lm}(X,F)$.
\end{enumerate}
\end{proposition}

\begin{proof}
Let $T\in\text{\rm se\-mi-K}(X,F)$. Then 
$T(B_X)\subseteq [-u,u]+\varepsilon B_F$ for some $u\in F_+$. 
As $T(B_X)$ is limited (\text{\rm a}-limited) 
if, for every $\varepsilon>0$, there exists a limited 
(resp. \text{\rm a}-limited) $A_\varepsilon\subseteq F$ such that 
$T(B_X)\subseteq A_\varepsilon+\varepsilon B_F$, 
it suffices to show that $[-u,u]$ is \text{\rm a}-limited (limited) in $F$. 

i) Let $(a_n)$ be a sequence in $[-u,u]$ and let $(f_n)$ be 
a disjoint $\text{\rm w}^\ast$-null sequence in $F'$. Since
$F\in\text{\rm (d)}$ then $(|f_n|)$ is $\text{\rm w}^\ast$-null.
It follows from $|f_n(a_n)|\le|f_n|(|a_n|)\le|f_n|(u)\to 0$ that
$f_n(a_n)\to 0$, and hence $[-u,u]$ is \text{\rm a}-limited
by Assertion \ref{a-limited and a-DP sets}~(i).

ii) This is similar to proof of part~(i) with the replacement of
Assertion \ref{a-limited and a-DP sets}~(i) by
Assertion \ref{limited and DP sets}~(i) and is omitted.
\end{proof}

\begin{proposition}\label{T is limited if T is MW and...}
If the lattice operations in $E'$ are sequentially $\text{\rm w}^\ast$-conti\-nuous,
then $\text{\rm MW}(E,F)\subseteq\text{\rm Lm}(E,F)$ for each $F$.
\end{proposition}

\begin{proof}
Let $T\in\text{\rm MW}(E,F)$ and let $(f_n)$ be $\text{\rm w}^\ast$-null 
in $F'$. Then $(T'f_n)$ be $\text{\rm w}^\ast$-null in $E'$. 
The sequential $\text{\rm w}^\ast$-continuity of lattice operations in $E'$ 
ensures that $(|T'f_n|)$ is $\text{\rm w}^\ast$-null in $E'$. 
Note that $\|f_n\|\le M\in\mathbb{R}$ for all $n\in\mathbb{N}$.
Let $(x_n)$ be disjoint bounded in $E_+$. Since $T\in\text{\rm MW}(E,F)$, 
then $(Tx_n)$ is norm null, and
$
   |T'(f_n)x_n|=|f_n(Tx_n)|\le\|f_n\|\cdot\|Tx_n\|\le M\|Tx_n\|\to 0.
$
By Assertion~\ref{Dodds-Fremlin}(ii), $(T'f_n)$ is norm null,
and hence $T$ is limited.
\end{proof}
\noindent
The next result gives a necessary and sufficient conditions on $F$ under which 
all continuous operators from $X$ to $F$ are \text{\rm a-}limited.

\begin{theorem}\label{dual Shur vs almost limited} 
{\rm 
Let $F$ be a Banach lattice. The following are equivalent. 
\begin{enumerate}[i)] 
\item 
$F\in\text{\rm (DDSP)}$.
\item 
$B_F$ is \text{\rm a}-limited.
\item 
The identity operator $I_F$ on $F$ is \text{\rm a}-limited.
\item
$\text{\rm a-Lm}(X,F)=\text{\rm L}(X,F)$ for every $X$.
\item
$\text{\rm a-Lm}(F)=\text{\rm L}(F)$.
\end{enumerate}}
\end{theorem}

\begin{proof} 
i)$\Longrightarrow$ii)
Let $(f_n)$ be a disjoint \text{\rm w}$^\ast$-null in $F'$.
Since $F\in\text{\rm (DDSP)}$, $(f_n)$ is norm null. Then
$(f_n)$ is uniformly null on $B_F$, and $B_F$
is \text{\rm a}-limited. 

ii)$\Longrightarrow$iii) 
It is obvious.

iii)$\Longrightarrow$i)
Let $(f_n)$ be disjoint \text{\rm w}$^\ast$-null in $F'$.
Since $I_F$ is \text{\rm a}-limited, then 
$(f_n)=((I_F)'f_n)$ is norm null. So, $F\in\text{\rm (DDSP)}$.

i)$\Longrightarrow$iv) 
It is enough to show that $\text{\rm L}(X,F)\subseteq\text{\rm a-Lm}(X,F)$. 
Let $T\in\text{\rm L}(X,F)$ and let $(f_n)$ be disjoint $\text{\rm w}^\ast$-null in $F'$. 
Since $F\in\text{\rm (DDSP)}$, then $(f_n)$ is norm null in $F'$, and hence
$(T'f_n)$ is norm null in $F'$, as required.

iv)$\Longrightarrow$v)$\Longrightarrow$iii)
It is obvious.
\end{proof}

\begin{proposition}\label{dprop then MW is alimited}
If $E\in\text{\rm (d)}$ then 
$\text{\rm r-MW}(E,F)\subseteq\text{\rm r-a-Lm}(E,F)$ for all $F$. 
\end{proposition}

\begin{proof}
It is enough to show $\text{\rm MW}_+(E,F)\subseteq\text{\rm a-Lm}_+(E,F)$.
Let $0\le T\in\text{\rm MW}(E,F)$ and let $(f_n)$ be disjoint  
$\text{\rm w}^\ast$-null in $E'$. Then $(|f_n|)$ is also $\text{\rm w}^\ast$-null
by the property (d). 
Since $|T'f_n|\le T'|f_n|$, the sequence $(|T'f_n|)$ 
is $\text{\rm w}^\ast$-null 
in $E'$. Thus, to show $\|T'f_n\|\to 0$, it suffices 
to show $T'f_n(x_n)\to 0$ for each disjoint bounded
$(x_n)$ in $E_+$. Consider
$
  |T'f_n(x_n)|=|f_n(Tx_n)|\le\|f_n\|\cdot\|Tx_n\|.
$
Since $\|f_n\|\le M$ for some $M$ and all $n$, and 
since $\|Tx_n\|\to 0$ as $T\in\text{\rm MW}(E,F)$, 
it follows $|T'f_n(x_n)|\to 0$. Thus, $\|T'f_n\|\to 0$, 
and hence $T\in\text{\rm a-Lm}(E,F)$.
\end{proof}

\subsection{Regularly almost limited operators.}
We need the following lemma.

\begin{lemma}\label{alimited are closed}
If $\text{\rm a-Lm}(X,F)\ni T_k\stackrel{\|\cdot\|}{\to}T\in\text{\rm L}(X,F)$
then $T\in\text{\rm a-Lm}(X,F)$.
\end{lemma}

\begin{proof}
Let $(f_n)$ be disjoint $\text{\rm w}^\ast$-null in $F'$.
We need to show that $(T'f_n)$ is norm null $X'$. Let $\varepsilon>0$.
Choose any $k$ with $\|T'_k-T'\|\le\varepsilon$. Since
$T_k\in\text{\rm a-Lm}(X,F)$, there exists $n_0$ such that
$\|T_k'f_n\|\le\varepsilon$ whenever $n\ge n_0$.
As $(f_n)$ is $\text{\rm w}^\ast$-null, there exists $M\in\mathbb{R}$
such that $\|f_n\|\le M$ for all $n\in\mathbb{N}$.
Then 
$$
   \|T'f_n\|\le\|T'_kf_n\|+\|T'_kf_n-T'f_n\|\le
   \varepsilon+\|T'_k-T'\|\|f_n\|\le\varepsilon+M\varepsilon
$$
for $n\ge n_0$. Since $\varepsilon>0$ is arbitrary
then $(T'f_n)$ is norm null, as desired.
\end{proof}
\noindent
The next result follows directly from Theorem \ref{P-norm} 
and Lemma \ref{alimited are closed}.

\begin{theorem}\label{P-norm in r-a-Lm}
For arbitrary Banach lattices $E$ and $F$, 
$\text{\rm r-a-Lm}(E,F)$ is a Banach space 
under the enveloping norm.
\end{theorem}
\noindent
Similarly, $\text{\rm r-Lm}(E,F)$, $\text{\rm r-o-Lm}(E,F)$,
$\text{\rm r-a-o-Lm}(E,F)$, $\text{\rm r-b-Lm}(E,F)$,
and\\ $\text{\rm r-a-b-Lm}(E,F)$
are all Banach space under their enveloping norms.

%%%%%%%%%%%%%%%%%%%%
\section{Almost Grothendieck operators.}
%%%%%%%%%%%%%%%%%%%%

This section is devoted to almost Grothendieck operators
introduced in \cite{GM}.

\subsection{Main definition.}
The Grothendieck property (the disjoint Grothen\-dieck property), 
that is considered as a property of the identity operator, 
motivates the following definition.

\begin{definition}\label{Main Grothendieck operators}
{\em A continuous operator
\begin{enumerate}[a)]
\item  
$T:X\to Y$ is called {\em Grothendieck} 
if $T'$ takes $\text{\rm w}^\ast$-null sequences of $Y'$
to \text{\rm w}-null sequences of $X'$$;$
\item 
$T:X\to F$ is called {\em almost Grothendieck} ($T$ is \text{\rm a-G})
if $T'$ takes disjoint $\text{\rm w}^\ast$-null sequences of $F'$ 
to \text{\rm w}-null sequences of $X'$ \cite[Def.3.1]{GM}.
\end{enumerate}}
\end{definition}
\noindent
If $T\in\text{\rm W}(X,Y)$ then
$(T'y'_n)$ is \text{\rm w}-null for each $\text{\rm w}^\ast$-null 
$(y'_n)$ in $Y'$ by \cite[Thm.5.23]{AlBu}, and hence 
$\text{\rm W}(X,Y)\subseteq\text{\rm G}(X,Y)$.
Clearly, $\text{\rm G}(X,F)\subseteq\text{\rm a-G}(X,F)$.
The identity operator $I:\ell^1\to\ell^1$ is \text{\rm a-G} yet not Grothendieck.
Since LW- and MW-operators 
are \text{\rm w}-compact, they are Gro\-then\-dieck and hence \text{\rm a-G}. 
If $E$ and $F$ are AM-spaces with strong order units $u_E$ and $u_F$, 
and $T:E\to F$ satisfies $Tu_E=u_F$, 
then $T'x'_n\xrightarrow[]{\text{\rm w}^\ast} 0$ in $E'$
for each disjoint ${\text{\rm w}^\ast}$-null sequence $(x'_n)$ in $F'$. 
In particular, $T'x'_n(u_E) = x'_n(Tu_E)\to 0$, and hence $\|T'x'_n\|\to 0$. 
Therefore $T'x'_n\xrightarrow[]{\text{\rm w}} 0$, and  $T$ is \text{\rm a-G}.
Definitions \ref{Main def of limited operators} and
\ref{Main Grothendieck operators} imply directly that 
$\text{\rm Lm}(X,Y)\subseteq\text{\rm G}(X,Y)$ and 
$\text{\rm a-Lm}(X,F)\subseteq\text{\rm a-G}(X,F)$.

\subsection{Conditions under which operators are almost Grothendieck.}
We begin with the following result. 

\begin{proposition}\label{d-prop implies semi is a-G}
If $F\in\text{\rm (d)}$ then $\text{\rm semi-K}(X,F)\subseteq\text{\rm a-G}(X,F)$.
\end{proposition}

\begin{proof}
Let $T\in\text{\rm semi-K}(X,F)$, and let $(g_n)$ be disjoint 
$\text{\rm w}^\ast$-null in $F'$. 
By the uniform boundedness principle, 
we can suppose $(g_n)\subset B_{F'}$. 
As $F\in\text{\rm (d)}$, $(|g_n|)$ is also $\text{\rm w}^\ast$-null. 
Since $T$ is semi-compact, 
for a given $\varepsilon>0$, there is $u\in F_+$ with
$\|(|Tx|-u)^+\|\le\frac{\varepsilon}{2}$ \text{for all}  $x\in B_X$.
As $(g_n)$ is $\text{\rm w}^\ast$-null, $|g_n|(u)\to 0$.
So, there is $n_\varepsilon$ such that $|g_n|(u)\le\frac{\varepsilon}{2}$
for all $n\ge n_\varepsilon$. Thus, if $x\in B_X$ then
$$
   |T'g_n(x)|= |g_n(Tx)|\le 
   |g_n| (|Tx|)\le |g_n| (|Tx|-u)^+ + |g_n|(|Tx|\wedge u)\le
$$
$$
   \|g_n\|\|(|Tx|-u)^+\| + |g_n|(u) 
   \le \varepsilon \ \
   \quad (\forall n \ge n_\varepsilon).
$$
Then $(T' g_n)$ in norm null, and hence it is ${\text{\rm w}^\ast}$-null.
\end{proof}
\noindent
Since LW-, and positive MW-operators are all semi-compact 
(\cite[Thm.5.71 and Thm.5.72]{AlBu}, \cite[3.6.14]{Mey}),
we obtain the following.

\begin{corollary}\label{MW and LW are a-G}
$\text{\rm LW}(E,F)\cup\text{\rm r-MW}(E,F)\subseteq\text{\rm a-G}(E,F)$.
\end{corollary}

\begin{theorem}\label{positive a-LW is a-G}
Let $E'$ be a KB-space and $F\in\text{\rm (d)}$. Then\\
$\text{\rm r-a-LW}(E,F)\subseteq \text{\rm r-a-G}(E,F)$.
\end{theorem}

\begin{proof}
It suffices to prove that each positive \text{\rm a-LW}-operator 
from $E$ to $F$ is \text{\rm a-G}. 
Let $0\le T\in\text{\rm a-LW}(E,F)$.
Then $T$ is continuous. 
In view of Definition \ref{Main Grothendieck operators}~b), 
we need to show that $\|T'f_n\|\to 0$ for each disjoint 
$\text{\rm w}^\ast$-null sequence $(f_n)$ in $F'$.
Let $(f_n)$ be disjoint $\text{\rm w}^\ast$-null in $F'$.
As $F\in\text{\rm (d)}$, $(|f_n|)$ is $\text{\rm w}^\ast$-null in $F'$. 
Since $T'$ is norm continuous, 
$T'$ is continuous when $F'$ and $E'$ are equipped 
with $\text{\rm w}^\ast$-topologies.
Thus, $(T'|f_n|)$ is $\text{\rm w}^\ast$-null in $E'$.
It follows from $|T'f_n|\le T'|f_n|$  that  
$(|T'f_n|)$ is $\text{\rm w}^\ast$-null in $E'$. 
By Assertion~\ref{Dodds-Fremlin}(ii), in order to show $\|T'f_n\|\to 0$, 
we need to check that $T'f_n(x_n)\to 0$ 
for each  disjoint bounded sequence $(x_n)$ in $E_+$.
Let $(x_n)$ be a disjoint bounded sequence in $E_+$. 
Since the norm in $E'$ is $\text{\rm o}$-continuous, 
$(x_n)$ is \text{\rm w}-null by Assertion \ref{E' is o-cont}.
Since $T$ is a-LW then $T'f_n(x_n)=f_n(Tx_n)\to 0$ e.g., 
by \cite[Thm.2.2]{BLM1}. Thus $\|T'f_n\|\to 0$ and 
hence $T$ is \text{\rm a-G}.
\end{proof}
\noindent
Since each Banach lattice with $\text{\rm o}$-continuous norm
is Dedekind complete, and hence has the property \text{\rm (d)}
by \cite[Prop.1.4]{Wnuk3}, we obtain the following.
\begin{corollary}\label{cor positive a-LW is a-G}
If both $E'$ and $F$ have $\text{\rm o}$-continuous norms then\\
$\text{\rm r-a-LW}(E,F)\subseteq \text{\rm r-a-G}(E,F)$.
\end{corollary}
\noindent
Now, we give necessary and sufficient conditions on $F$ under which 
every continuous $T:X\to F$ is almost Grothendieck.

\begin{theorem}\label{DGP vs a-G} 
{\rm 
Let $F$ be a Banach lattice. The following are equivalent. 
\begin{enumerate}[i)] 
\item 
$F\in\text{\rm (DGP)}$.
\item 
The identity operator $I_F$ on $F$ is \text{\rm a-G}.
\item
$\text{\rm a-G}(X,F)=\text{\rm L}(X,F)$ for every $X$.
\item
$\text{\rm a-G}(F)=\text{\rm L}(F)$.
\end{enumerate}}
\end{theorem}

\begin{proof} 
i)$\Longleftrightarrow$ii)
Let $(f_n)$ be disjoint \text{\rm w}$^\ast$-null in $F'$.
The condition $I_F\in\text{\rm a-G}(F)$ means 
$((I_F)'(f_n))=(f_n)$ is \text{\rm w}-null, which 
it is equivalent to $F\in\text{\rm (DGP)}$. 

ii)$\Longrightarrow$iii)
Let $T\in\text{\rm L}(X,F)$ and let $(f_n)$ be 
disjoint \text{\rm w}$^\ast$-null in $F'$.
Since $I_F$ is \text{\rm a-G},
$((I_F)'(f_n))=(f_n)$ is \text{\rm w}-null.
As $T$ is \text{\rm w}-continuous,
$(Tf_n)$ is \text{\rm w}-null, and hence 
$T\in\text{\rm a-G}(X,F)$.

iii)$\Longrightarrow$iv)$\Longrightarrow$ii)
It is obvious.
\end{proof}

\subsection{Regularly (almost) Grothendieck operators.}
We continue with the following lemma.

\begin{lemma}\label{a-G are closed}
If $\text{\rm a-G}(X,F)\ni T_k\stackrel{\|\cdot\|}{\to}T\in\text{\rm L}(X,F)$
then $T\in\text{\rm a-G}(X,F)$.
\end{lemma}

\begin{proof}
Let $(f_n)$ be disjoint $\text{\rm w}^\ast$-null in $F'$.
We need to show that $(T'f_n)$ is \text{\rm w}-null in $X'$. 
So, take any $g\in F''$. Let $\varepsilon>0$.
Pick a $k$ with $\|T'_k-T'\|\le\varepsilon$. Since
$T_k\in\text{\rm a-G}(X,F)$, there exists $n_0$ such that
$|g(T_k'f_n)|\le\varepsilon$ whenever $n\ge n_0$.
Note that $\|f_n\|\le M$ for some $M\in\mathbb{R}$
and for all $n\in\mathbb{N}$. Since $\varepsilon>0$
is arbitrary, it follows from
$$
   |g(T'f_n)|\le|g(T_k'f_n-T'f_n)|+|g(T'f_n)|\le
   \|g\|\|T_k'-T'\|\|f_n\|+\varepsilon\le(\|g\|M+1)\varepsilon
$$
for $n\ge n_0$, that $g(T'f_n)\to 0$. Since $g\in F''$ is arbitrary,
$T\in\text{\rm a-G}(X,F)$.
\end{proof}
\noindent
The next result follows from Theorem \ref{P-norm} and Lemma \ref{a-G are closed}.

\begin{theorem}\label{P-norm in a-G}
For arbitrary Banach lattices $E$ and $F$, 
$\text{\rm r-a-G}(E,F)$ is a Banach space 
under the enveloping norm.
\end{theorem}
\noindent
Similarly, it can be shown that
$\text{\rm r-G}(E,F)$ is a Banach space 
under its enveloping norm.

%%%%%%%%%%%%%%%%%%%%
\section{Almost Dunford--Pettis operators}
%%%%%%%%%%%%%%%%%%%%

In this section, we investigate several modifications of Dunford--Pettis opeators
(cf. \cite{San,AqBo,AE,BLM1,CCJ,BLM2,AEW}) and related topics.

\subsection{$\text{DP}^\ast$-properties.} 
Recall that a Banach space $X$ has 
\begin{enumerate}[1)]
\item
the {\em Dunford--Pettis property} (or, $X\in\text{\rm (DPP)}$)
if $f_n(x_n)\to 0$ for each \text{\rm w}-null $(f_n)$ in $X'$
and each \text{\rm w}-null $(x_n)$ in $X$;
\item
the {\em Dunford--Pettis$^\ast$ property} 
(or, $X\in\text{\rm (DP}^\ast\text{\rm P)}$) if $f_n(x_n)\to 0$ 
for each $\text{\rm w}^\ast$-null $(f_n)$ in $X'$ and
each \text{\rm w}-null $(x_n)$ in $X$.
\end{enumerate}
\noindent
Obviously, $\text{\rm (DP}^\ast\text{\rm P)}\Longrightarrow\text{\rm (DPP)}$. 
The following property is often taking as 
the definition of $\text{\rm (DP}^\ast\text{\rm P)}$:
\begin{enumerate}[\text{2'})]
\item 
$X\in\text{\rm (DP}^\ast\text{\rm P)}$ iff 
all relatively \text{\rm w}-compact subsets of $X$ are limited.
\end{enumerate}
Combinig \text{2'}) with \text{d}) of Definition \ref{Main Schur property} gives 
$$
   X\in\text{\rm (DP}^\ast\text{\rm P)}\cup\text{\rm (GPP)}
  \Longleftrightarrow\dim(X)<\infty.
$$ 

\begin{definition}\label{Dunford--Pettis properties}
{\em A Banach lattice $E$ has 
\begin{enumerate}[a)]
\item
the {\em almost Dunford--Pettis$^\ast$ property} 
(or, $E\in\text{\rm (a-DP}^\ast\text{\rm P)}$) if
$f_n(x_n)\to 0$ for each $\text{\rm w}^\ast$-null $(f_n)$ in $E'_+$
and each disjoint \text{\rm w}-null $(x_n)$ in $E$;
\item
the {\em weak Dunford--Pettis$^\ast$ property} 
(or, $E\in\text{\rm (wDP}^\ast\text{\rm P)}$)
if all relatively \text{\rm w}-compact subsets of $E$ are \text{\rm a}-limited
\cite[Def.3.1]{CCJ}.
\end{enumerate}}
\end{definition}
\noindent
For example, $\ell^1\in\text{\rm (wDP}^\ast\text{\rm P)}$.
It should be clear that 
$\text{\rm (bi-sP)}\Rightarrow\text{\rm (a-DP}^\ast\text{\rm P)}$.
By Assertion \ref{disj w-null is mod w-null},
$E\in\text{\rm (a-DP}^\ast\text{\rm P)}$ iff
$f_n(x_n)\to 0$ for each $\text{\rm w}^\ast$-null $(f_n)$ in $E'_+$
and each disjoint \text{\rm w}-null $(x_n)$ in $E_+$.
We include a proof of the following characterization 
of $\text{\rm (a-DP}^\ast\text{\rm P)}$
as we did not find appropriate references. 

\begin{assertion}\label{X has Dunford--Pettis* property}
The following holds.
\begin{enumerate}[\em (i)]
\item 
$X\in\text{\rm (DP}^\ast\text{\rm P)}$
iff $\lim\limits_{n\to\infty}f_n(x_n)=0$ 
for each $\text{\rm w}^\ast$-convergent $(f_n)$ in $X'$
and each {\em w}-null $(x_n)$ in $X$$;$
\item 
$E\in\text{\rm (a-DP}^\ast\text{\rm P)}$
iff $\lim\limits_{n\to\infty}f_n(x_n)=0$ 
for each $\text{\rm w}^\ast$-convergent $(f_n)$ in $E'$
and each disjoint \text{\rm w}-null $(x_n)$ in $E$.
\end{enumerate}
\end{assertion}

\begin{proof}
(i) The sufficiency is trivial. For the necessity, 
let $x_n\stackrel{\text{\rm w}}{\to}0$ and 
$f_n \stackrel{\text{\rm w}^\ast}{\to}f$. 
Since $(f_n -f)\stackrel{\text{\rm w}^\ast}{\to}0$, and 
hence $(f_n-f)(x_n)\to 0$, and as $f(x_n)\to 0$, 
we obtain $f_n(x_n)\to 0$.

(ii) is similar to (i).
\end{proof}

\subsection{(Weak) Dunford--Pettis operators.}  
We begin with the following.

\begin{definition}\label{def of DP operators} 
{\em An operator $T:X\to Y$ is called:
\begin{enumerate}[a)]
\item {\em Dunford--Pettis} or {\em completely continuous} 
(shortly, $T$ is \text{\rm DP}) 
if $T$ takes \text{\rm w}-null sequences to 
norm null ones (cf. \cite[p.340]{AlBu}).
\item {\em weak Dun\-ford--Pet\-tis} (shortly, $T$ is \text{\rm wDP})
if $f_n(Tx_n)\to 0$ whenever $(f_n)$ is \text{\rm w}-null in $Y'$ and 
$(x_n)$ is \text{\rm w}-null in $X$ \cite[p.349]{AlBu}.
\end{enumerate}}
\end{definition}
\noindent
Looking at an operator form the point of view of
redistributing a property of a space between the doman and 
the range of the operator as in \cite{AEG2}, 
we observe that 
\begin{enumerate}[1)]
\item 
\text{\rm DP}-opeators appear via the Schur property 
and equially deserve the name of 
\text{\rm (SP)}-operators, whereas
\item 
\text{\rm wDP}-opeators appear via the next property: $X\in\text{\rm (wDP)}$ if
$f_n(x_n)\to 0$ for all \text{\rm w}-null $(f_n)$ in $X'$ and 
all \text{\rm w}-null $(x_n)$ in $X$.
\end{enumerate}
\noindent
Clearly, $\text{\rm K}(X,Y)\subseteq\text{\rm DP}(X,Y) 
\subseteq\text{\rm wDP}(X,Y)\subseteq \text{\rm L}(X,Y)$.
By \cite[Thm.5.79]{AlBu},
$T$ is \text{\rm DP} iff $T$ takes \text{\rm w}-Cauchy sequences 
to norm convergent sequences.
By the Rosenthal theorem (cf. \cite[Thm.4.72]{AlBu}), 
any bounded sequence $(x_n)$ in $X$ 
has a subsequence $(x_{n_k})$ which is either \text{\rm w}-Cauchy or, 
alternatively, is equivalent 
to the standard basis of $\ell^1$. 
It follows that {\em each \text{\rm DP}-operator 
$X\stackrel{T}{\to}Y$ is compact
if $\ell^1$ does not embed in $X$} \cite[Thm.5.80]{AlBu}.
The identity operator $I:\ell^1\to\ell^1$ is \text{\rm DP} 
but its adjoint $I:\ell^\infty\to\ell^\infty$ is not \text{\rm DP}.
It is well known that  the operator $L^1[0,1]\stackrel{T}{\to}\ell^\infty$, 
$Tf=\left(\int_0^1 f(t)r^+_k(t)\,dt \right)_{k=1}^\infty$ is \text{\rm wDP} 
yet not \text{\rm DP}, as $r_n\stackrel{\text{\rm w}}{\to}0$ and 
$\|Tr_n\|\ge \int_0^1 r_n(t)r^+_n(t)\,dt\equiv\frac{1}{2}$.
An operator $X\stackrel{T}{\to} Y$ is \text{\rm wDP} 
iff $S\circ T$ is \text{\rm DP} 
for every \text{\rm w}-compact operator $S$ from $Y$ to an arbitrary 
$Z$ (cf. \cite[Thm.5.81,Thm.5.99]{AlBu}).

\subsection{Almost (weak) Dunford--Pettis operators.}

\begin{definition}\label{Main def of a-DP operators} 
{\em An operator $T:E\to Y$ is called$:$
\begin{enumerate}[a)]
\item 
{\em almost Dunford--Pettis} (shortly, $T$ is \text{\rm a-DP}) 
if $T$ takes disjoint \text{\rm w}-null sequences to norm null ones \cite{San,Wnuk2};
\item 
{\em almost weak Dun\-ford--Pet\-tis} (shortly, $T$ is \text{\rm a-wDP}) 
if $f_n(Tx_n) \to 0$ whenever $(f_n)$ is \text{\rm w}-null in $Y'$ and 
$(x_n)$ is disjoint \text{\rm w}-null in $E$. 
\end{enumerate}}
\end{definition}
\noindent
It follows from Assertion \ref{disj w-null is mod w-null} (cf. also \cite[Lm.4.1]{AEW})
that $T\in\text{\rm a-DP}(E,Y)$ (resp. $T\in\text{\rm a-wDP}(E,Y)$) 
iff $T$ takes disjoint \text{\rm w}-null sequences of $E_+$ to norm null ones
(resp. $f_n(Tx_n) \to 0$ whenever $(f_n)$ is \text{\rm w}-null in $Y'$ and 
$(x_n)$ is disjoint \text{\rm w}-null in $E_+$). Like in the observation after 
Definition \ref{def of DP operators}, it is worth noting that
\begin{enumerate}[1)]
\item 
\text{\rm a-DP}-operators appear via an equivalent form of the positive Schur property 
and probably deserve also the name of \text{\rm (PSP)}-operators.
\item 
\text{\rm a-wDP}-operators appear via the following property
of Banach lattices: $E\in\text{\rm (a-wDP)}$ whenever
$f_n(x_n)\to 0$ for every \text{\rm w}-null $(f_n)$ in $E'$ and 
every disjoint \text{\rm w}-null $(x_n)$ in $E$. It should be also clear that
\text{\rm (bi-sP)}$\Rightarrow$\text{\rm (a-wDP)}.
\end{enumerate}
\noindent
Obviously, $\text{\rm DP}(X,Y)\subseteq\text{\rm wDP}(X,Y)$, 
$\text{\rm DP}(E,Y)\subseteq\text{\rm a-DP}(E,Y)$, and
$$
   \text{\rm DP}(E,F)\subseteq\text{\rm wDP}(E,F)\subseteq
   \text{\rm a-wDP}(E,F)\subseteq\text{\rm L}(E,F).
$$
By \cite[Thm.4.1]{AqBo}, 
$\text{\rm a-DP}(E,F)=\text{\rm DP}(E,F)$ for all $F$ 
iff the lattice operations in $E$ are sequentially \text{\rm w}-continuous.
The identity operator: 
\begin{enumerate}[]
\item
$I:L^1[0,1]\to L^1[0,1]$ is \text{\rm a-DP} yet not \text{\rm DP};
\item
$I:c\to c$ is \text{\rm a-wDP} yet neither \text{\rm wDP} nor \text{\rm a-DP}.
\end{enumerate}
\noindent
Observe that $\text{\rm DP}(E,X)\subseteq\text{\rm MW}(E,X)$ 
whenever the norm in $E'$ is \text{\rm o}-continu\-ous. 
Indeed, let $T\in\text{\rm DP}(E,X)$ and let 
$(x_n)$ be a disjoint bounded sequence in $E$. 
By Assertion \ref{E' is o-cont}, $(x_n)$ is \text{\rm w}-null. 
As $T$ is \text{\rm DP}, $\|Tx_n\|\to 0$ and hence $T\in\text{\rm MW}(E,X)$. 
Notice that the condition of \text{\rm o}-continuity for norm in $E'$ is essential here, 
as $I_{\ell^1}\in\text{\rm DP}(\ell^1)$ due to 
the Schur property in $\ell^1$, 
yet $I_{\ell^1}\not\in\text{\rm MW}(\ell^1)$ 
because the disjoint bounded 
sequence $(e_n)$ of unit vectors of $\ell^1$ is not norm null.
The domination property for \text{\rm a-DP}-operators 
was established in \cite[Cor.2.3]{AE}. The proofs of the 
next proposition is a straightforward 
modification of the proofs of the Kalton--Saab domination 
theorem (cf. \cite[Thm.5.101]{AlBu}.

\begin{proposition}\label{dominated by awDP is awDP}
Any positive operator dominated by an \text{\rm a-wDP}-operator 
is likewise an \text{\rm a-wDP}-operator.
\end{proposition}
\noindent
We also omit the straightforward proof of the following fact.

\begin{lemma}\label{DP are closed}
Let $E$ and $F$ be Banach lattices.
\begin{enumerate}[{\rm i)}]
\item
If $\text{\rm a-DP}(E,F)\ni T_k\stackrel{\|\cdot\|}{\to}T\in\text{\rm L}(E,F)$
then $T\in\text{\rm a-DP}(E,F)$.
\item
If $\text{\rm a-wDP}(E,F)\ni T_k\stackrel{\|\cdot\|}{\to}T\in\text{\rm L}(E,F)$
then $T\in\text{\rm a-wDP}(E,F)$.
\end{enumerate}
\end{lemma}

\noindent
In view of Lemma \ref{DP are closed},
the next result follows directly from Theorem \ref{P-norm}. 

\begin{theorem}\label{P-norm in DP}
For arbitrary Banach lattices $E$ and $F$, 
$\text{\rm r-a-DP}(E,F)$ and $\text{\rm r-a-wDP}(E,F)$
are both Banach spaces under their enveloping norms.
\end{theorem}
\noindent
The similar result yields for $\text{\rm r-DP}(E,F)$ and $\text{\rm r-wDP}(E,F)$.

%\subsection{Some conditions under which operators are \text{\rm a-DP}.}

\begin{proposition}\label{MW, a-DP, a-LW}
$\text{\rm MW}(E,Y)\subseteq\text{\rm a-DP}(E,Y)$.
\end{proposition}

\begin{proof}
Let $S\in\text{\rm MW}(E,Y)$. 
Assertion \ref{Meyer-Nieberg}~(i) implies $S'\in\text{\rm LW}(Y',E')$,
which yields $S'$ is semi-compact due to \cite[Thm.5.71]{AlBu}.
Thus, $S\in\text{\rm a-DP}(E,Y)$ by \cite[Thm.4.3]{AqBo}.
\end{proof}

\begin{assertion}\label{(a-LW)+ is in a-DP}
{\em (see, \cite[Prop.2.3]{BLM1})}
$\text{\rm a-LW}_+(E,F)\subseteq\text{\rm r-a-DP}_+(E,F)$.
\end{assertion}
\noindent
We include a proof of the following version 
of Assertion \ref{(a-LW)+ is in a-DP}. 

\begin{theorem}\label{(a-MW)' is a-DP}
If $T\in \text{\rm r-a-MW}(E,F)$ then $T'\in\text{\rm r-a-DP}(F',E')$.
\end{theorem}

\begin{proof}
It is enough to prove that $T'$ is \text{\rm a-DP} for each 
positive a-MW-operator $T:E\to F$. 
Let $0\le T\in\text{\rm a-MW}(E,F)$ and let $(f_n)$ be 
a disjoint \text{\rm w}-null sequence in $F'$. 
Then $(|f_n|)$ is disjoint \text{\rm w}-null in $F'$ 
by Assertion \ref{disj w-null is mod w-null}.
Since $T$ is a-MW then $f_n(Tx_n)=(T'f_n)(x_n)\to 0$ for each 
disjoint bounded sequence $(x_n)$ in $E$. 
By Assertion~\ref{Dodds-Fremlin}(ii), for proving $\|T'f_n\|\to 0$,
it remains to show that $(|T'f_n|)$ is \text{\rm w}$^\ast$-null.  
Since $T':F'\to E'$ is positive, $T'$ is continuous, 
and hence is \text{\rm w}-continuous. 
Thus, $(T'|f_n|)$ is \text{\rm w}-null, 
and hence is \text{\rm w}$^\ast$-null. 
The inequality $|T'f_n|\le T'|f_n|$ implies 
that $(|T'f_n|)$ is \text{\rm w}$^\ast$-null, as desired.
\end{proof}

\subsection{Some conditions on operators to be \text{\rm\bf a-G}.} 
\begin{proposition}\label{positive a-DP is a-G}
Let $E'$ have $\text{\rm o}$-continuous norm and $F\in\text{\rm (DGP)}$. Then
$$
    \text{\rm r-a-DP}(E,F)\subseteq\text{\rm r-a-G}(E,F).
$$
\end{proposition}

\begin{proof}
It suffices to prove  
$\text{\rm a-DP}_+(E,F)\subseteq\text{\rm a-G}(E,F)$. 
Let $T\in\text{\rm a-DP}_+(E,F)$.
By \cite[Thm.5.1]{AEW}, $T'\in\text{\rm a-DP}(F',E')$.
Let $(f_n)$ be disjoint $\text{\rm w}^\ast$-null in $F'$. 
In order to prove $T\in\text{\rm a-G}(E,F)$, it remains to show that 
$(T'f_n)$ is \text{\rm w}-null in $E'$. Since $F\in\text{\rm (DGP)}$
then $(f_n)$ is disjoint $\text{\rm w}$-null in $F'$. Since 
$T'\in\text{\rm a-DP}(F',E')$ then $(T'f_n)$ is norm null
and hence is \text{\rm w}-null, as desired.
\end{proof}

\begin{proposition}\label{when r-a-G is L-r}
Let $E\in\text{\rm (a-DP}^\ast\text{\rm P)}$ and $F\in\text{\rm (d)}$. Then 
$$
   \text{\rm r-a-G}(E,F)={\cal L}_r(E,F).
$$
\end{proposition}

\begin{proof}
It suffices to prove that each positive operator $T:E\to F$ is \text{\rm a-G}. 
Let $0\le T\in{\cal L}(E,F)$,
and let $(f_n)$ be disjoint $\text{\rm w}^\ast$-null in $F'$.
Since $F\in\text{\rm (d)}$ then $(|f_n|)$ is also $\text{\rm w}^\ast$-null.
As $T':F'\to E'$ is positive, $T'$ is continuous and hence $T'$ is 
$\text{\rm w}^\ast$-continuous. 
Therefore $(T'|f_n|)$ is $\text{\rm w}^\ast$-null in $E'$.
It follows from $|T'f_n|\le T'|f_n|$ that
$(|T'f_n|)$ is $\text{\rm w}^\ast$-null. Since 
$E\in\text{\rm (a-DP}^\ast\text{\rm P)}$, 
$|T'f_n|(x_n)\to 0$ for each disjoint bounded $(x_n)$ in $E$.
By Assertion~\ref{Dodds-Fremlin}(ii), $\|T'f_n\|\to 0$  and hence 
$(T'f_n)$ is $\text{\rm w}$-null. Therefore $T$ is \text{\rm a-G}.
\end{proof}

\subsection{Miscellanea.}
Recall that a collection ${\cal P}$ of operators is said to be {\em injective} 
whenever $JT\in {\cal P}$
for $T\in {\cal P}$ and every isometric homomorphism $J$.
Let $(x_n)$ be a disjoint bounded sequence in $E$. 
Then $\|JTx_n\|=\|Tx_n\|\to 0$ if $T$ is \text{\rm MW} then so it is $JT$.
In particular, \text{\rm MW}, \text{\rm a-DP}, and \text{\rm o}-weakly 
compact operators are injective.

\begin{proposition}\label{DPO is oW}
Every \text{\rm DP}-operator $T:E\to Y$ is \text{\rm o}-weakly compact.
\end{proposition}

\begin{proof}
Let $(x_n)$ be a disjoint order bounded in $E$. 
Then $x_n\stackrel{\rm w}{\to}0$, and as $T$ is \text{\rm DP}, 
$\|Tx_n\|\to 0$. 
This yields $T$ is o-weakly compact by \cite[Thm.5.57]{AlBu}.
\end{proof}
\noindent
Theorem \ref{DPO is oW} extends to \text{\rm a-DP}-operators
as follows.

\begin{proposition}\label{extends to a-DPO}
Every \text{\rm a-DP}-operator $T:E\to X$ is \text{\rm b}-weakly compact. 
\end{proposition}

\begin{proof}
Let $(x_n)$ be a disjoint \text{\rm b}-bounded sequence in $E$. 
Then $x_n\stackrel{\rm w}{\to}0$. Since $T$ is \text{\rm a-DP}, $(Tx_n)$
is norm null and $T$ is \text{\rm b}-weakly compact.
\end{proof}

\begin{proposition}\label{if DPO T is limited...}
If \ $\text{\rm DP}(E,F)\subseteq \text{\rm Lm}(E,F)$, 
then either $E'$ has $\text{\rm o}$-continu\-ous norm,
or every order bounded subset of $F$ is limited.
\end{proposition}

\begin{proof}
It is enough to show that if the norm in $E'$ is not 
$\text{\rm o}$-continuous then, for each  $y\in F_+$, the order interval
$[0,y]$ is limited, that is $g_n(y_n)\to 0$ for each sequence $(y_n)$ in $[0,y]$ and 
each $\text{\rm w}^\ast$-null sequence $(g_n)$ in $F'$. Let the norm in $E'$ 
be not $\text{\rm o}$-continuous and $y\in F_+$.
There is a positive order bounded $(x'_n)$ in $E'$ 
with $\|x'_n\|=1$ for all $n\in{\mathbb N}$. Let $0\le x'_n \le x'$. 
Let $T:E\to\ell^1$ be defined as $Tx:=(x'_n(x))$. Since 
$
   \sum_{n=1}^\infty |x'_n(x)|\le\sum_{n=1}^\infty x'_n(|x|)\le 
   x'|x|\quad\text{for}\quad x\in E,
$
$T$ is well defined and takes values in $\ell^1$. Let $S:\ell^1\to F$ be defined as 
$S(\alpha_n):=\sum_{n=1}^\infty \alpha_n y_n$. Denote $U=S\circ T$, so
$Ux=\sum_{n=1}^\infty x'_n(x)y_n$ for $x\in E$. 
The operator $U$ is \text{\rm DP}. By the assumption, $U$ is limited. Observe that
$$
   |U'g_n|=\sum_{n=1}^\infty|g_n(y_i)|\cdot x'_i\ge|g_n(y_n)|\cdot x'_n\ge 0
$$
and
$
  |g_n(y_n)|=\|x'_n\|\cdot|g_n(y_n)|\le\|T'g_n\|\to 0.
$
Thus $[0,y]$ is limited.
\end{proof}

\begin{theorem}\label{when a-DPO has o-cont norm}
Let the norms in $E'$ and $F$ be \text{\rm o}-continuous. 
Then each order bounded \text{\rm a-DP}-operator $T:E\to F$ 
has \text{\rm o}-continuous norm.
\end{theorem}

\begin{proof}
Let $T: E\to F$ be an order bounded \text{\rm a-DP}-operator. 
Let $(x_n)$ be a disjoint bounded sequence in $E$.
Then $(x_n)$ is \text{\rm w}-null by Assertion \ref{E' is o-cont}.
Since $T$ is \text{\rm a-DP}, $\|Tx_n\|\to 0$ and $T$ is MW. 
As $F$ has o-continuous norm,
$T$ is also LW. Thus by \cite[Thm.5.68]{AlBu}, $T$ has o-continuous norm. 
\end{proof}
\noindent
Following \cite[p.330]{AlBu}, we denote 
$$
  {\cal A}_T=\{S\in{\cal L}_{ob}(E,F):\exists n
  \in{\mathbb N}\quad\text{with}\quad |S|\le n|T|\}
$$
and
$$
  \text{\rm Ring}(T)=\text{cl}_{\|\cdot\|}\left(\{S\in\text{\rm L}(X,Y): 
  S=\sum^n_{i=1}R_iTS_i; S_i\in\text{\rm L}(X), R_i\in\text{\rm L}(Y)\}\right).
$$
We conclude with the following application of Theorem~\ref{when a-DPO has o-cont norm}. 

\begin{corollary}\label{pro Ring}
Let the norms in $E'$ and $F$ be \text{\rm o}-continuous, and let $E$ be either 
Dedekind $\sigma$-complete 
or have a quasi-interior point. Then ${\cal A}_T\subseteq \text{\rm Ring}(T)$
for each $T\in\text{\rm a-DP}_+(E,F)$.
\end{corollary}

\begin{proof}
By Theorem~\ref{when a-DPO has o-cont norm}, $T$ has order continuous norm. 
The rest follows from \cite[Thm.5.70]{AlBu}.
\end{proof}

{\tiny 
%%%%%%%%%%%%%%%%%%%%
}

\end{document}